\documentclass[11pt]{article}

\usepackage[latin1]{inputenc}
\usepackage[T1]{fontenc}
\usepackage{color}
\usepackage{eurosym}
\usepackage[top=4cm,bottom=4cm,left=3cm,right=3cm]{geometry}
\usepackage{fancyhdr}
\usepackage{float}
\usepackage{amssymb,amsmath, amsthm}
\usepackage{url,graphicx}

\usepackage{booktabs}
\usepackage{stmaryrd} 
\usepackage{mathenv}
\usepackage{dsfont}
\usepackage{amssymb}
\usepackage{latexsym}
\usepackage{amsmath}
\usepackage{color}
\usepackage{comment}
\usepackage{amsthm}
\usepackage{bm}
\usepackage{bbm} 
\usepackage[bbgreekl]{mathbbol} 
\usepackage{authblk}
\usepackage{multirow}
\usepackage{enumerate} 
\usepackage{ragged2e}

 \usepackage{setspace}
\newif\ifrs
\rstrue
\ifrs \usepackage{mathrsfs} \fi  
\newif\ifcol
\coltrue 

\newtheorem{theorem*}{Theorem}[section]
\newtheorem{ass*}[theorem*]{Assumption}
\newtheorem{note*}[theorem*]{Note}
\newtheorem{lemma*}[theorem*]{Lemma}
\newtheorem{definition*}[theorem*]{Definition}
\newtheorem{proposition*}[theorem*]{Proposition}
\newtheorem{corollary*}[theorem*]{Corollary}
\newtheorem{remark*}[theorem*]{Remark}
\newtheorem{example*}[theorem*]{Example}

\numberwithin{equation}{section}
\newif\ifcol
\coltrue
\ifcol
\newcommand{\colorr}{\color[rgb]{0.8,0,0}}

\newcommand{\colorn}{\color[rgb]{1,1,1}}

\else

\newcommand{\colorr}{\color{black}}

\newcommand{\colorn}{\color{black}}

\fi
\excludecomment{en-text}
\includecomment{jp-text}
\includecomment{comment}
\def\bd{\begin{description}}
\def\ed{\end{description}}

\def\D2{\bbD_{2,\infty-}}

\def\D{{\bf D}}
\def\E{{\bf E}}
\def\F{{\bf F}}

\def\U{{\bf U}}

\def\calb{{\cal B}}

\def\cale{{\cal E}}
\def\calf{{\cal F}}

\def\cali{{\cal I}}

\def\call{{\cal L}}

\def\caln{{\cal N}}

\def\calq{{\cal Q}}

\def\cals{{\cal S}}
\def\calt{{\cal T}}

\def\calx{{\cal X}}

\def\be{\begin{equation}}
\def\ee{\end{equation}}
\def\bea{\begin{eqnarray}}
\def\eea{\end{eqnarray}}
\def\beas{\begin{eqnarray*}}
\def\eeas{\end{eqnarray*}}
\def\bi{\begin{itemize}}
\def\ei{\end{itemize}}
\def\im{\item}
\def\bd{\begin{description}}
\def\ed{\end{description}}
\def\l{\left}
\def\r{\right}

\newcommand{\bbD}{{\mathbb D}}

\newcommand{\reels}{\mathbb{R}}
\newcommand{\naturels}{\mathbb{N}}

\newcommand{\esp}{\mathbb{E}}
\newcommand{\proba}{\mathbb{P}}

\newcommand{\x}{\mathbf{x}}
\newcommand{\boldv}{\mathbf{v}}
\newcommand{\bt}{\mathbf{t}}

\date{} 

\begin{document}
 \title{Quasi-likelihood analysis for marked point processes and application to marked Hawkes processes}
 \author{Simon Clinet\footnote{Faculty of Economics, Keio University. 2-15-45 Mita, Minato-ku, Tokyo, 108-8345, Japan. E-mail: clinet@keio.jp. On behalf of all authors, the corresponding author states that there is no conflict of interest.
}}

\maketitle 

\begin{abstract}
We develop a quasi-likelihood analysis procedure for a general class of multivariate marked point processes. As a by-product of the general method, we establish under stability and ergodicity conditions the local asymptotic normality of the quasi-log likelihood, along with the convergence of moments of quasi-likelihood and quasi-Bayesian estimators. To illustrate the general approach, we then turn our attention to a class of multivariate marked Hawkes processes with generalized exponential kernels, comprising among others the so-called Erlang kernels. We provide explicit conditions on the kernel functions and the mark dynamics under which a certain transformation of the original process is Markovian and $V$-geometrically ergodic. We finally prove that the latter result, which is of interest in its own right, constitutes the key ingredient to show that the generalized exponential Hawkes process falls under the scope of application of the quasi-likelihood analysis. 
\end{abstract}

\section{Introduction}

We introduce a general class of $d$-dimensional marked point processes (MPP) observed on the real half-line, and represented by a family of integer-valued random measures $\overline{N} = (\overline{N}^\alpha)_{\alpha=1,...,d}$ on $\reels_+ \times \mathbb{X}$. Assuming that the compensating measure $\nu^\alpha$ of $\overline{N}^\alpha$ is parametrized in the following form 
\bea \label{eqCompensatorIntro} 
\nu^\alpha(ds,dx,\theta) = f^\alpha(s,x,\theta)ds  \rho(dx) 
\eea
where $\theta \in \Theta \subset \mathbb{R}^n$, and for some measure $\rho(dx)$, we define the quasi-likelihood process for $T \geq 0$ as
\bea \label{likelihoodIntro}
l_T(\theta) = \sum_{\alpha=1}^d\int_{[0,T] \times \mathbb{X}} \textnormal{log} f^\alpha(s,x,\theta) \overline{N}^\alpha(ds,dx) - \sum_{\alpha=1}^d\int_{[0,T] \times \mathbb{X}} f^\alpha(s,x,\theta) ds  \rho(dx).
\eea  
When the horizon time $T \to +\infty$, and under ergodicity assumptions, our first concern is to carry out the so-called quasi-likelihood analysis (QLA), which consists in deriving a polynomial large deviation inequality for the field $l_T$, along with its local asymptotic normality (LAN). As a by-product, we establish the convergence of moments of the quasi-maximum likelihood estimator (QMLE) and of any quasi-Bayesian estimator (QBE), that is, if $\theta_T$ is such an estimator, then 
\bea \label{convMomentIntro}
\esp\l[u\l(\sqrt{T}(\theta_T- \theta^*)\r)\r] \to \esp\l[u\l(\Gamma^{-1/2} \xi\r)\r],
\eea 
where $u$ is any continuous function of polynomial growth, $\theta^*$ is the unknown parameter, $\Gamma \in \mathbb{R}^{ n \times n}$ is the asymptotic Fisher information and $\xi$ follows an $n$-dimensional standard normal distribution. \\

\smallskip 

There has been a growing literature on QLA over the past decade, for a large variety of stochastic processes. \cite{YoshidaPolynomial2011} first established general criteria for the derivation of polynomial large deviation inequalities, extending arguments from \cite{ibragimov1981statistical} and \cite{kutoyants1984parameter}. Subsequently, QLA has been successfully applied to a large panel of processes\footnote{a detailed list of applications of the QLA can be found in the introduction of \cite{yoshida2018partial}.}. In particular, general point processes were tackled in \cite{clinet2017statistical} (see also \cite{muni2020analyzing} for Cox-type models). Accordingly, our first contribution follows this line of research. To our knowledge, this is the first work that provides conditions for the convergence of moments (\ref{convMomentIntro}) for \textit{marked} point processes. Our main findings, in line with the conditions previously proposed in \cite{clinet2017statistical}, is that the key ingredient for the QLA in the context of an MPP is a family of laws of large numbers for certain transformations of the density $f^\alpha$, uniformly in $\theta$, and having a minimal rate of convergence in $T^\gamma$ for some $\gamma \in (0,1/2)$. These convergences are typically obtained for Markovian or exponentially mixing stationary processes although such conditions are not necessary in general.\\ 

\smallskip

Let us emphasize that, in light of convergence (\ref{convMomentIntro}), the QLA method provides stronger results than the standard asymptotic normality of QMLE and QBE, at the cost of stronger moment and ergodicity conditions than what is usually required. In return, the moment convergence (\ref{convMomentIntro}) paves the way to a large body of applications including, among others, information criteria-based model selection (see, e.g \cite{umezu2019aic,eguchi2018schwarz}) and local parametric estimation (see \cite{clinet2018statistical} for a locally parametric exponential Hawkes process).   \\

\smallskip

This first contribution can be naturally put into perspective with the existing literature on maximum likelihood estimation for point processes. In the unmarked case, the seminal work of \cite{OgataAsymptotic1978} proved the asymptotic normality of the maximum likelihood estimator (MLE) for ergodic stationary point processes. \cite{kutoyants1984parameter} (Theorems 4.5.5 and 4.5.6) provided general conditions for LAN and moment convergence of the MLE and Bayesian estimators (BE). Finally, \cite{clinet2017statistical} carried out the QLA in an ergodic framework. Other parametric methods can be consulted in \cite{karr1991point} and \cite{tuan1981estimation}. For marked point processes, \cite{nishiyama1995local} derived LAN results, although under the very restrictive assumption that $f(t,x,\theta) =y_t a(x,\theta)$ where $(y_t)_{t \geq 0}$ is known and predictable. Apart from specific structures (e.g. birth-and-death processes \cite{keiding1975maximum}), general results for the MLE have yet to be established.  \\ 

\smallskip 

Our second concern is the application of the QLA method to a class of generalized exponential kernel-based marked Hawkes processes, which extends the exponential Hawkes process in several ways, while retaining its strong mixing properties. There is a vast literature on the Hawkes process introduced in \cite{hawkes1971spectra}, with applications ranging from seismology \cite{ogata1988statistical}, to sociology \cite{crane2008robust} and finance \cite{bacry2015hawkes}. Recently, and in particular in financial data analysis, marked point processes models have emerged, including marked Hawkes processes as in \cite{fauth2012modeling}, \cite{rambaldi2017role}, \cite{morariu2018state}, \cite{clinet2019asymptotic}, \cite{richards2019score} and \cite{wu2019queue}. Letting $\lambda^\alpha(t,\theta)$ be the stochastic intensity of the \textit{counting process} $N^\alpha = \overline{N}^\alpha(\mathbb{X} \times \cdot)$ and $X_{t-}$ the mark of the last jump before $t$, we will say that $\overline{N}$ is a marked Hawkes process if
\bea \label{HawkesIntensityIntro} 
\lambda^\alpha(t,\theta ) = \phi_\alpha\l( \l(\int_{[0,t) \times \mathbb{X}} h_{\alpha\beta}(t-s, x,\theta) \overline{N}_\beta(ds, dx)\r)_{\beta =1,...,d}, X_{t-}, \theta\r). 
\eea 
Here, $h_{\alpha\beta}$ are the marked excitation kernels, and $\phi_\alpha$ are typically sub-linear functions in their first two arguments. We therefore allow the mark to affect both the shape of the cross-excitation functions $h_{\alpha\beta}$, and the baseline intensity through the dependence of $\phi_\alpha$ in $X_{t-}$. \\

\medskip

In light of our first contribution, applying the QLA method to such a process amounts to finding adequate conditions ensuring that a process satisfying (\ref{HawkesIntensityIntro}) is ergodic with fast rate $T^\gamma$. Lately, a few papers focused on the establishment of strong ergodicity and mixing results for the exponential Hawkes process by taking advantage of its Markovian structure. \cite{abergel2015long} gave a linear Lyapounov function for the multivariate exponential Hawkes process; \cite{clinet2017statistical} derived the full $V$-geometric ergodicity of the process for an exponential Lyapounov function, when the spectral radius of the associated excitation matrix is smaller than unity. \cite{duarte2019stability} tackled the case of an Erlang kernel in a univariate framework; Finally \cite{dion2020exponential} dealt with non-linear exponential kernel Hawkes process with possible inhibition. In this work, we build on this strand of research for the marked case. Our main focus lies in the situation where the kernels are of the form 
\bea \label{expoRepresnetationIntro}
h_{\alpha\beta}(s, x, \theta) = \langle A_{\alpha\beta}(\theta) | e^{-sB_{\alpha\beta}(\theta)}\rangle g_{\alpha\beta}(x,\theta)
\eea
where $A_{\alpha\beta}(\theta), B_{\alpha\beta}(\theta) \in \mathbb{R}^{p \times p}$ for some $p \geq 1$, $\langle \cdot | \cdot \rangle$ denotes the usual inner product on $\mathbb{R}^{p \times p}$, and $g_{\alpha\beta}$ accounts for the impact of the mark process on the cross-excitation. The temporal term in (\ref{expoRepresnetationIntro}) generalizes the aforementioned models, and we prove that a function has the shape $ \langle A | e^{ - \cdot B} \rangle$ and is integrable if and only if it is a linear combination of terms of the form $P\cdot C \cdot e^{-r \cdot}$, where $P$ is polynomial, $C$ is a sine or a cosine, and $r>0$. We finally focus on the case where the dynamics of the mark process follows a transition kernel yielding a Markovian representation of the marked point process. \\

\smallskip

Our second contribution goes as follows. First, we provide adequate assumptions on the mark transition kernel and on $h_{\alpha\beta}$ that ensure $V$-geometric ergodicity of the process. When the marks are independent and identically distributed, this condition reduces to the standard assumption on the spectral radius of the excitation matrix, integrated over the mark space. When the marks admit a more intricate dynamics, the condition is more restrictive. We find in particular that the mark process itself needs to satisfy a certain drift condition. This is briefly illustrated on the queue-reactive Hawkes model which has recently been considered in \cite{wu2019queue}. The novelty of this contribution lies in the general form of the temporal component in (\ref{expoRepresnetationIntro}) along with the presence of marks in the excitation kernel. Second, we make use of this strong mixing property to show that the QLA applies to the generalized exponential marked Hawkes process when $\phi_\alpha$ is linear in the excitations. \\

\smallskip

The paper is structured as follows. Section 2 details the QLA for an ergodic marked point process. Section 3 introduces the marked Hawkes process and establishes its $V$-geometric ergodicity when its excitation kernel is matrix-exponential, under well-chosen stability and non-degeneracy assumptions. Section 4 shows the application of the QLA to the generalized exponential marked Hawkes process. We conclude in Section 5. Proofs and technical details are relegated to the Appendix.  

\section{Quasi-likelihood analysis for marked point processes: a general framework} \label{sectionQLA}

In this section, we introduce our general statistical framework. Our main results are Theorem \ref{thmLAN} and Theorem \ref{thmQLA2} which respectively prove the local asymptotic normality of the quasi-likelihood process and the convergence of moments of the QMLE and the QBE.

\subsection*{Notation for the marked point process}
We assume that we are given a stochastic basis $\calb=(\Omega,\calf,\F,\proba)$ with a filtration $\F = (\calf_t)_{t \in \reels_+}$, that contains all the observed processes necessary to the statistical inference. The marked point process is defined as follows.

\begin{definition*} \label{defMPP}
Assume the existence for $\alpha = 1,...,d$, $d \in \naturels - \{0\}$ of a sequence of couples $(T_i^\alpha,X_i^\alpha)_{i \geq 1}$ such that 

\begin{itemize} 
\item The $T_i^\alpha$'s are $\F$-stopping times, with
$$0 < T_1^\alpha < \cdots < T_N^\alpha < \cdots < +\infty \textnormal{, } \proba-\textnormal{a.s}$$
and $T_N^\alpha \to^{a.s} +\infty$ when $N \to + \infty$. 
\item The $X_i^\alpha$'s are $\calf_{T_i^\alpha}$-measurable random variables taking values in the measurable space $(\mathbb{X},\calx)$.
\end{itemize}
We call multivariate marked point process $\overline{N} = (\overline{N}^\alpha)_{\alpha = 1,...,d}$ the family of random measures on $\reels_+ \times \mathbb{X}$, adapted to $\F$, and defined by $\overline{N}^\alpha(ds,dx) = \sum_{i=0}^{+\infty} \delta_{(T_i^\alpha,X_i^\alpha)}(ds,dx)$.
\end{definition*}  

Hereafter, we will denote by $\nu^\alpha(ds,dx)$ the so-called compensator of $\overline{N}^\alpha(ds,dx)$, that is the unique predictable measure such that $\overline{N}^\alpha([0,t] \times A) - \nu^X([0,t] \times A)$ is a local martingale for any set $A \in \calx$ (\cite{JacodLimit2003}, Theorem 1.8). \\
\smallskip 

It is convenient to associate to $\overline{N}$ a \textit{counting process} $N$, which is simply defined for any $\alpha \in \{1,...,d\}$ as $N_t^\alpha = \overline{N}^\alpha([0,t] \times \mathbb{X})$, $t \in \reels_+$, and its compensator $\Lambda_t^\alpha = \nu^\alpha([0,t] \times \mathbb{X})$. Hereafter, we will always assume that the covariates of $N$ do not have common jump times: for $\alpha,\beta \in \{1,...,d\}$ with $\alpha \neq \beta$, $\Delta N^\alpha \Delta N^\beta = 0 \textnormal{, } \proba-\textnormal{a.s.}$ In what follows, we will often use the notations $\tilde{M}^\alpha = \overline{N}^\alpha - \nu^\alpha$ and $\tilde{N}^\alpha = N^\alpha - \Lambda^\alpha$ which are both local martingale measures.\\

\smallskip

We finally turn our attention to the parametrization of the marked point process proposed in (\ref{eqCompensatorIntro}). We assume that there exists $\theta^* \in \Theta \subset \reels^n$ for some $n \in \naturels$ and where $\Theta$ is convex, open and bounded\footnote{We also always assume that $\Theta$ is regular enough so that the Sobolev embedding Theorem holds (see \cite{adams2003sobolev}, Theorem 4.12,  p. 85)}, $\rho$ a measure on $\mathbb{X}$ and $f^\alpha$ a non-negative function such that 
$$\nu^\alpha(ds,dx) = f^\alpha(s,x,\theta^*) ds  \rho(dx), \alpha \in \{1,...,d\}$$
and we gather those densities in the vector $f = (f^1,...,f^d)$. It will be convenient to put for $\alpha \in \{1,...,d\}$ and $t \geq 0$
\beas 
\lambda^{\alpha}(t,\theta) = \int_{\mathbb{X}} f^\alpha(t,x,\theta) \rho(dx)
\eeas 
and 
\bea \label{decompoLambdaQ} 
q^{\alpha}(t,x,\theta) = \lambda^{\alpha}(t,\theta)^{-1} \mathbb{1}_{\{\lambda^{\alpha}(t,\theta) \neq 0\}} f^\alpha(t,x,\theta).
\eea 
The above decomposition allows us to see $\lambda^\alpha$ as the stochastic intensity of $N^\alpha$, which encodes the dynamics of the waiting times through the informal equality $\esp[dN_t^\alpha|\calf_{t-}] = \lambda^\alpha(t,\theta)dt$, whereas $q^\alpha(t,\cdot, \theta)$ can be seen as as the density of a mark on the covariate $\alpha$ at time $t$ given $\calf_{t-}$. In particular, as soon as $\lambda^{\alpha}(t,\theta) \neq 0$, we have $\int_{\mathbb{X}} q^\alpha(t,x,\theta)\rho(dx) = 1$.

\subsection*{Quasi-Likelihood and main assumptions}
We start with mild conditions on $f$ that ensure the existence of the quasi-log likelihood process below.

\bd
\im[[A1\!\!]]
The mapping $f : \Omega \times \reels_+ \times \mathbb{X} \times \Theta \to \reels_+^d$ is $\calf \otimes \mathbf{B}(\reels_+) \otimes \calx \otimes \mathbf{B}(\Theta)$-measurable. Moreover, almost surely,
  \bd
		\im[{\bf (i)}] for any $\theta \in \Theta$, $f(\cdot,\cdot,\theta)$ is a predictable function on $\Omega \times \reels_+ \times \mathbb{X}$.
		\im[{\bf (ii)}] for any $(s,x) \in \reels_+ \times \mathbb{X}$, $\theta \to f(s,x,\theta)$ is in $C^4(\Theta)$, and admits a continuous extension to $\overline{\Theta}$.
		\im[{\bf (iii)}] For any $\theta \in \Theta$, $\alpha \in \{1,...,d\}$, $f^\alpha(t,x,\theta) = 0 $ if and only if  $f^\alpha(t,x, \theta^{*}) =0, dt \rho(dx)$-a.e.
	\ed
\ed

The quasi log-likelihood process is defined at time $T \in \reels_+$ and point $\theta \in \Theta$ as   
\bea \label{eqLogLik}
l_T(\theta) = \sum_{\alpha=1}^d\int_{[0,T] \times \mathbb{X}} \textnormal{log} f^\alpha(s,x,\theta) \overline{N}^\alpha(ds,dx) - \sum_{\alpha=1}^d\int_{[0,T] \times \mathbb{X}} f^\alpha(s,x,\theta) ds  \rho(dx).
\eea 
Note that $l_T$ is indeed (up to a constant term) the log-likelihood process related to $\overline{N}$ as soon as $\F$ is the canonical filtration associated to $\overline{N}$ (\cite{JacodLimit2003}, Theorem 5.43). However, in the present case, $\F$ and $l_T$ may involve additional explanatory processes with no restrictions (apart from the fact that they must be observable so that $l_T$ can be computed in practice). We call quasi maximum likelihood estimator any quantity $\hat{\theta}_T$ satisfying
\bea
\hat{\theta}_T \in \textnormal{argmax}_{\theta \in \Theta} l_T(\theta)
\eea 
if the right-hand side is non-empty. For a given continuous prior density $p : \Theta \to \reels_+$ satisfying $0 < \inf_{\theta \in \Theta} p(\theta) \leq \sup_{\theta \in \Theta} p(\theta) < +\infty$, we also define the associated Bayesian estimator
\bea 
\tilde{\theta}_T = \frac{\int_\Theta \theta \textnormal{exp}(l_T(\theta)) p(\theta)d\theta}{\int_\Theta \textnormal{exp}(l_T(\theta)) p(\theta)d\theta}.  
\eea 
 In light of (\ref{decompoLambdaQ}), (\ref{eqLogLik}) can be rewritten as
\bea 
l_T(\theta) = \underbrace{\sum_{\alpha=1}^d \int_0^T\textnormal{log} \lambda^\alpha(t,\theta)dN_t^\alpha -\sum_{\alpha=1}^d \int_0^T \lambda^\alpha(t,\theta)dt}_{l_T^{(1)}(\theta)} + \underbrace{\sum_{\alpha=1}^d\int_{[0,T] \times \mathbb{X}} \textnormal{log} q^\alpha(s,x,\theta) \overline{N}^\alpha(ds,dx)}_{l_T^{(2)}(\theta)},
\eea 
where $l_T^{(1)}(\theta)$ accounts for the contribution of the waiting times and $l_T^{(2)}(\theta)$ accounts for that of the marks to the global log-likelihood process. The next assumption gives standard moment and smoothness conditions that ensure the desired large deviation inequality on $\textnormal{exp}(l_T)$ necessary for the convergence of moments of the QMLE and the QBE. We use the following notations. For a vector or a matrix $x$, $|x| = \sum_{i} |x_i|$. Moreover, $\|X\|_p = \esp[|X|^p]^{1/p}$.

\bd
\im[[A2\!\!]]
The density process $f$ and its first derivatives in $\theta$ satisfy, for any $p > 1$, for any $\alpha \in \{1,...,d\}$
	\bd
		\im[(i)] $\sup_{t \in \reels_+}  \sum_{i=0}^3{\left\| \sup_{\theta \in \Theta}\l|\partial_\theta^i \lambda^\alpha(t,\theta)\r|\right\|_p} <+\infty,$
		\im[(ii)]  $\sup_{t \in \reels_+} \left\| \sup_{\theta \in \Theta}\l|\lambda^\alpha(t,\theta)^{-1}\r| \mathbb{1
		}_{\{\lambda^\alpha(t,\theta) \neq 0\}}\right\|_p <+\infty,$
		\im[(iii)] $\sup_{t \in \reels_+}  \sum_{i=0}^3 \int_\mathbb{X} \esp \l[\sup_{\theta \in \Theta} \l|\partial_\theta^i \textnormal{log}q^\alpha(t,x,\theta)\r|^p   q^\alpha(t,x,\theta^*)\r]\rho(dx)  <+\infty,$
		\im[(iv)] $\sup_{t \in \reels_+}  \sum_{i=0}^3 \int_\mathbb{X} \esp \l[\sup_{\theta \in \Theta} \l|\partial_\theta^i \textnormal{log}q^\alpha(t,x,\theta)\r|^{-p} q^\alpha(t,x,\theta^*) \r] \rho(dx)  <+\infty.$
	\ed
\ed

Crucial to our analysis, and standard in the literature on QLA (see e.g Assumption [A6] in \cite{YoshidaPolynomial2011}, Assumption [A3] in \cite{clinet2017statistical}), is the uniform convergence of the scaled process $T^{-1}l_T(\theta)$ to a limit $\mathbb{Y}(\theta)$ with a rate of convergence of the shape $T^{\gamma}$ for some $\gamma \in (0,1/2)$. The next condition aims to provide such a property by assuming the $\mathbb{L}^p$ convergences of particular time averaged transformations of $\lambda$ and $q$ with a minimal rate of convergence. In what follows, we let $E = \reels_+ \times \reels_+ \times \reels^n$, and we define $D_{\uparrow}(E,\reels)$ as the set of functions $\phi : E \to \reels$ such that: (i) $\phi$ is of class $C^1$ on $(\reels_+ - \{0\}) \times (\reels_+ - \{0\}) \times \reels^n$; (ii) For $(u,v,w) \in E$, $\phi$ and $|\nabla \phi|$ are of polynomial growth in $(u,v,w,\frac{\mathbb{1}_{\{u \neq 0\}}}{u},\frac{\mathbb{1}_{\{v \neq 0\}}}{v})$; (iii) $\phi(0,v,w) = \phi(u,0,w) = 0$.

\bd
\im[[A3\!\!]]
There exists $\gamma \in (0,1/2)$, there exist $\pi_\alpha : D_{\uparrow}(E,\reels) \times \Theta \to \reels$ and $\chi_\alpha : \{0,1,2\} \times \Theta \to \reels$ such that for any $\phi \in D_{\uparrow}(E,\reels)$, for any $p \geq 1$ 
\beas 
\sup_{\theta \in \Theta} T^{\gamma} \l\|T^{-1} \int_0^T \phi(\lambda^\alpha(s,\theta^*), \lambda^\alpha(s,\theta), \partial_\theta \lambda^\alpha(s,\theta))ds - \pi_\alpha(\phi,\theta)\r\|_p \to 0,
\eeas 
and for $k \in \{0,1,2\}$
\beas 
\sup_{\theta \in \Theta} T^{\gamma} \l\|T^{-1} \int_0^T \int_{\mathbb{X}} \partial_\theta^{k} \textnormal{log}q^\alpha(s,x,\theta) q^\alpha(s,x,\theta^*)\rho(dx) \lambda^\alpha(s,\theta^*)ds - \chi_\alpha(k,\theta)\r\|_p \to 0.
\eeas


\ed

Condition \textbf{[A3]} essentially assumes the ergodicity of the couple $(\lambda,q)$, uniformly in $\theta \in \Theta$. Note also that the exponent $\gamma$ in the rate of convergence should be positive, and in general cannot be larger than or equal to $1/2$ since the above laws of large numbers often also satisfy central limit theorems in $\sqrt{T}$. A fundamental example where \textbf{[A3]} can be established is the case where the couple $(\lambda,q)$ is mixing with a polynomial rate $t^{-\epsilon}$ where $\epsilon > 2\gamma/(1-2\gamma)$ (see \cite{clinet2017statistical}, \textbf{[M2]} and Lemma 3.16, p. 1815). This is the strategy that we will adopt when proving \textbf{[A3]} for the marked Hawkes process considered in Section \ref{sectionQLAHawkes}. It is worth mentioning that in principle, as illustrated in \cite{clinet2017statistical}, Examples 3.2-3.6, very different families of models enjoy such laws of large numbers, which is what makes \textbf{[A3]} quite general. 

\subsection*{Asymptotic properties}

We are now ready to derive the asymptotic properties of the log-likelihood field. In the following, we always extend the $\textnormal{log}$ function to $0$ by putting $\textnormal{log}(0) = 0$, for notational simplicity. Put for $i \in \{1,2\}$

\bea 
\mathbb{Y}_T^{(i)}(\theta) = T^{-1}(l_T^{(i)}(\theta) - l_T^{(i)}(\theta^*)),
\eea 
\bea
\Delta_T^{(i)} = T^{-1/2} \partial_\theta l_T^{(i)}(\theta^*)
\eea 
and 
\bea \label{gammaRepresentation} 
\Gamma_T^{(i)}(\theta) = - T^{-1} \partial_{\theta}^2 l_T^{(i)}(\theta),
\eea 
which, after some algebraic manipulations yield the expressions
\beas  
\mathbb{Y}_T^{(1)}(\theta) = T^{-1}\sum_{\alpha=1}^d \int_0^T \textnormal{log}\frac{\lambda^\alpha(t,\theta)}{\lambda^\alpha(t,\theta^*)}dN_t^\alpha -T^{-1}\sum_{\alpha=1}^d \int_0^T \l[\lambda^\alpha(t,\theta)-\lambda^\alpha(t,\theta^*)\r]dt,
\eeas  
\beas 
\mathbb{Y}_T^{(2)}(\theta)  = T^{-1}\sum_{\alpha=1}^d\int_{[0,T] \times \mathbb{X}} \textnormal{log} \frac{q^\alpha(t,x,\theta)}{q^\alpha(t,x,\theta^*)} \overline{N}^\alpha(dt,dx),
\eeas 
\beas  
 \Delta_T^{(1)} = T^{-1/2}\sum_{\alpha=1}^d\int_0^T  \frac{\partial_\theta \lambda^\alpha(t, \theta^*)}{\lambda^\alpha(t,\theta^*)} \mathbb{1}_{\{\lambda^\alpha(t,\theta^*) \neq 0\}} d\tilde{N}_t^\alpha,  
\eeas  
\beas 
 \Delta_T^{(2)} = T^{-1/2}\sum_{\alpha=1}^d\int_{[0,T] \times \mathbb{X}}   \partial_\theta \textnormal{log} q^\alpha(t,x, \theta^*)\mathbb{1}_{\{q^\alpha(t,x,\theta^*) \neq 0\}} \tilde{M}^\alpha(dt,dx), 
\eeas 
where we have used that that $\int_{\mathbb{X}} q(t,x,\theta) \rho(dx) = 1$, and at point $\theta^*$ 
\beas 
\Gamma_T^{(1)}(\theta^*) = -T^{-1}\sum_{\alpha=1}^d\int_0^T \partial_\theta^2 \textnormal{log} \lambda^\alpha(t,\theta^*)d\tilde{N}_t^\alpha + T^{-1}\sum_{\alpha=1}^d\int_0^T \frac{\partial_\theta \lambda(t, \theta^*)^{\otimes 2}}{\lambda^\alpha(t,\theta^*)} \mathbb{1}_{\{\lambda^\alpha(t,\theta^*) \neq 0\}}dt,  
\eeas  
and
\beas 
\Gamma_T^{(2)} (\theta^*)= -T^{-1}\sum_{\alpha=1}^d \int_{[0,T] \times \mathbb{X}} \partial_{\theta}^2 \textnormal{log} q^\alpha(t,x,\theta^*) \overline{N}^\alpha(dt,dx), 
\eeas 
where for a vector $x \in \reels^n$, $x^{\otimes 2} = x\cdot x^T \in \reels^{n \times n}$. 

\begin{lemma*} \label{lemY}
Assume \textnormal{\textbf{[A1]-[A3]}}. There exists $\mathbb{Y}(\theta)$ such that for any $p \geq 1$, as $T \to +\infty$
$$  T^\gamma \|\sup_{\theta \in \Theta} |\mathbb{Y}_T(\theta) - \mathbb{Y}(\theta)| \|_p \to 0.$$
\end{lemma*}

Before we move to the LAN property, it is necessary to assume a standard identifiability condition on the limit field $\mathbb{Y}$ obtained in the previous lemma.
\bd
\im[[A4\!\!]]
We have $\textnormal{inf}_{\theta \in \Theta - \{\theta^*\}} -\frac{\mathbb{Y}(\theta)}{|\theta - \theta^*|^2} > 0.$
\ed

\begin{remark*} \label{rmkGamma}
Note that the limit field $\mathbb{Y}$ is automatically of class $C^2$ on $\Theta$ by the uniform convergences coming from $\textnormal{\textbf{[A3]}}$ and Lemma \ref{lemAppliSobolev}. Moreover, writing $\Gamma = - \partial_\theta^2 \mathbb{Y}(\theta^*)$ the asymptotic Fisher information matrix, note that \textnormal{\textbf{[A4]}} is equivalent to 
$$\textnormal{sup}_{\theta \in \Theta - \{\theta^*\}} \mathbb{Y}(\theta) <0 \textnormal{  and  } \Gamma >0.$$
\end{remark*}

\begin{lemma*} \label{lemGamma}
Assume \textnormal{\textbf{[A1]-[A4]}}. Let $\Gamma \in \reels^{n \times n}$ be the positive matrix defined in Remark \ref{rmkGamma}. Then, for any random ball $B_T \subset \Theta$ such that $\textnormal{diam} (B_T) \to^\proba 0$, we have as $T \to +\infty$
$$ \sup_{\theta \in B_T} |\Gamma_T(\theta) - \Gamma| \to^\proba 0.$$
Moreover, for any $p \geq 1$, we have 
$$ \sup_{T \in \reels_+} T^\gamma \l\| \Gamma_T(\theta^*) - \Gamma \r\|_{p} < +\infty.$$
\end{lemma*}

\begin{remark*}
It is immediate to see from (\ref{gammaRepresentation}) that $\Gamma$ admits the representation
\bea \label{fisherRepresentation}
\Gamma = \proba-\lim_{T \to +\infty} \frac{1}{T} \sum_{\alpha=1}^d \int_{[0,T] \times \mathbb{X}} \frac{\partial_\theta f^\alpha(t,x, \theta^*)^{\otimes 2}}{f^\alpha(t,x,\theta^*)} \mathbb{1}_{\{f^\alpha(t,x,\theta^*) \neq 0\}}  dt \rho(dx)  
\eea
where we recall that for $x \in \reels^n$, $x^{\otimes 2} = x\cdot x^T \in \reels^{n \times n}$.
\end{remark*}
The next lemma gives the asymptotic distribution of the scaled score process $(\Delta_{uT})_{u \in [0,1]
}$.
\begin{lemma*} \label{lemDelta}
We have the convergence in distribution for the Skorokhod topology $\mathbb{D}([0,1])$, as $T \to +\infty$
    $$ \l(\sqrt{u}\Delta_{uT}\r)_{u \in [0,1]} \to^d \Gamma^{1/2} (W_{u})_{u \in [0,1]},$$
    where $W$ is a standard Brownian motion on $[0,1]$. Moreover, we have for any $p \geq 1$
$$ \sup_{T \in \reels_+} \| \Delta_T \|_p < +\infty. $$
\end{lemma*}

We are now ready to state the local asymptotic normality of the likelihood field along with two deviations inequalities. First, for $u \in U_T = \{ u \in \reels^n | \theta^* + \frac{u}{\sqrt T} \in \Theta\}$, we take 
\bea 
\mathbb{Z}_T(u)= \textnormal{exp}\l[l_T\l( \theta^* + \frac{u}{\sqrt T}\r) - l_T(\theta^*)\r]
\eea 
and we extend the domain of $\mathbb{Z}_T$ to $\reels^n$ by taking $\mathbb{Z}_T$ continuously tending to $0$ as $|u| \to + \infty$ outside $U_T$. The next lemma shows that $\mathbb{Z}_T$ converges in distribution to the limit field
\bea 
\mathbb{Z}(u) = \textnormal{exp}\l( u^T\Delta  - u^T \Gamma u\r)
\eea 
where $\Delta \sim \caln(0, \Gamma)$.
\begin{theorem*} \label{thmLAN}
Assume \textnormal{\textbf{[A1]-[A4]}}. The following holds:
\bd
     \im[(i)] (Polynomial type large deviation inequality) For any $L >0$, there exists $C_L \geq 0$ such that $\sup_{r >0} \sup_{T >0}\proba \l[\sup_{u \in U_T, |u|>r} \mathbb{Z}_T(u) \geq e^{-r}\r] \leq \frac{C_L}{r^L}$.
    \im[(ii)] (LAN) $\mathbb{Z}_T \to^d \mathbb{Z}$ where the convergence happens in the space of continuous functions decreasing to $0$ as $|u| \to + \infty$ endowed with the uniform topology. 
    \im[(iii)] (Inverse moment condition) There exists $\delta >0$ such that $\sup_{T \in \reels_+} \esp \l[ \l( \int_{u| |u| \leq \delta} \mathbb{Z}_T(u)du\r)^{-1} \r] < +\infty$.
\ed
\end{theorem*}

The main consequence of the previous theorem is the convergence of moments for the QMLE and the QBE.
\begin{theorem*} \label{thmQLA2}
Assume \textnormal{\textbf{[A1]-[A4]}}. We have, for any continuous function $u$ of polynomial growth, 
$$ \esp[u(\sqrt T (\hat{\theta}_T - \theta^*))] \to \esp[u(\Gamma^{-1/2} \xi)],$$
$$ \esp[u(\sqrt T (\tilde{\theta}_T - \theta^*))] \to \esp[u(\Gamma^{-1/2} \xi)]$$
where $\xi$ follows an $n$-dimensional standard normal distribution.
\end{theorem*}

\section{Generalized exponential marked Hawkes process} \label{sectionHawkes}

We introduce in this section a new class of marked Hawkes processes enjoying a Markovian representation, and we show that under reasonable stability conditions theses processes are ergodic and even exponentially mixing. The main result of this section is Theorem \ref{thmVgeom}, which is the main tool necessary for the application of QLA to the generalized exponential marked Hawkes process. In this section, we drop the dependence of all quantities in $\theta$ since no parametrization is necessary for our purpose. All processes can be thought of as taken at point $\theta^*$. We return to the problem of statistical inference in Section \ref{sectionQLAHawkes}.   

\subsection*{Model definition and first properties}

We consider the marked point process $\overline{N} = (\overline{N}^\alpha)_{\alpha = 1,...,d}$ of Section \ref{sectionQLA}. We let $T_1 < T_2 < ...$ be the sequence of jump times associated to the global counting process $\sum_{\alpha = 1}^d N^\alpha$. We also consider that all the covariates are affected by the same (possibly multidimensional) mark sequence $(X_i)_{i \geq 1}$. It will be convenient to associate to this sequence of marks a piecewise constant and right continuous process defined by $X_{t} = X_i$ for $T_i \leq t < T_{i+1}$. In this section and the next, we assume that $\mathbb{X}$ is the subset of a finite-dimensional normed space and such that the topology inherited from the associated norm $|\cdot|$ makes $\mathbb{X}$ locally compact and separable. $\calx$ is then naturally taken as the associated Borel $\sigma$-field. Typical examples are when the state-space is $\reels^q$, $\mathbb{Z}^q$ or any categorical finite set. We will say that $\overline{N}$ is a (multivariate) marked Hawkes process if for any $\alpha \in \{1,...,d\}$, the stochastic intensity of the associated counting process $N^\alpha$ has the form
\bea \label{generalIntensityhab}
\lambda^\alpha(t ) = \phi_\alpha\l( \l(\int_{[0,t) \times \mathbb{X}} h_{\alpha\beta}(t-s, x) \overline{N}_\beta(ds, dx)\r)_{\beta =1,...,d}, X_{t-}\r), 
\eea
where $\phi_\alpha$ is continuous and $h_{\alpha\beta}$ is measurable, $\phi_\alpha : \reels_+^{d} \times \mathbb{X} \to \reels_+$ is not necessarily linear in its first argument, and for $\alpha,\beta \in \{1,...,d\}$, $h_{\alpha\beta}: \reels_+ \times \mathbb{X} \to \reels_+$ corresponds to the excitation kernel impacting $\lambda^\alpha$ every time a jump on the covariate $N^\beta$ occurs. Note that we exclude self-inhibition mechanisms as $h_{\alpha\beta}$ is required to yield non-negative values only. When $\phi_\alpha(u,x) = \nu_\alpha(x)$ does not depend on $u$, $\lambda^\alpha$ is piecewise constant and purely driven by the mark process $X$. Most classes of pure-jump Markov processes are thus comprised in this representation. In particular, popular models in finance such as the zero-intelligence model of \cite{AbergelLOB2013} or the queue-reactive model of \cite{RosenbaumQueue2014} fall under the scope of this model. When $\phi_\alpha(u,x) = \nu_\alpha(x) + \sum_{\beta = 1}^d u_\beta$, the marked Hawkes process is said linear as studied in, e.g \cite{BremaudStability1996} in the unmarked case. Hereafter, we keep a general form for $\phi_\alpha$ (although our ergodicity results will come at the cost of sub-linearity), but we restrict ourselves to the case where the kernels admit the following multiplicative representation:
\bd
\im[[E\!\!]] For any $\alpha,\beta \in \{1,...,d\}$, we have 
\bea \label{defHgenExp}
h_{\alpha\beta}(s,x)= \langle A_{\alpha\beta} | e^{-s B_{\alpha\beta}} \rangle g_{\alpha\beta}(x),
\eea 
where $A_{\alpha\beta} \in \reels^{p \times p}, B_{\alpha\beta}  \in \reels^{p \times p}$, and $g_{\alpha\beta}(x) \in \reels_+$ for some $p \geq 1$, and where $\langle\cdot|\cdot\rangle$ denotes the canonical inner product on $\mathbb{R}^{p \times p}$.
\ed

In (\ref{defHgenExp}), the inner product corresponds to the temporal component of the kernel, and has a generalized exponential shape that we will characterize later on. Hereafter, $B_{\alpha\beta}$ is always assumed to have eigenvalues with positive real parts. We will extensively use this condition when proving the stability of the process. Although apparently simple, the above matrix-exponential representation yields more general functions than the pure exponential kernel. The function $g_{\alpha\beta}$ accounts for the impact of the marks on the excitation process. Such a multiplicative form, already used in \cite{clinet2019asymptotic} and \cite{richards2019score}, makes most calculations tractable while retaining the dual impact of time and marks on the excitation kernel. In \cite{richards2019score}, which fits the above model on financial limit order book data, linear and quadratic shapes have been proposed for the boost function $g_{\alpha\beta}$, with several mark processes ranging from trade volumes, price transformations, market imbalance, to transformations of the counting process $N$ itself (similarly to queue-reactive models). Details can be found in the aforementioned paper, Table 2, p. 28.\\

\smallskip 

We call generalized exponential kernel any function $h_{\alpha\beta}$ satisfying (\ref{defHgenExp}), and generalized exponential marked Hawkes process (GEMHP) a process whose excitation kernels admit all the representation (\ref{defHgenExp}). 
It turns out that the following fundamental example
\bea \label{cosexpRep} 
f_P(s) = P(s)e^{-rs}, 
\eea 
also called Erlang kernel, where $P$ is a polynomial function of the form $P(s) = \sum_{k=1}^p a_k s^k$, and $r>0$ is a fixed constant yields such a representation. Indeed defining $D = r \cali_{p \times p}$ where $\cali_{p \times p}$ is the indentity matrix having dimension $p$, $-M$ the nilpotent matrix such that $-M_{i,j} = \mathbb{1}_{\{j = i+1\}}$ for $1 \leq i,j \leq p$, we readily verify that
\bea \label{polynomExample} 
\textnormal{exp}(-t(D+M)) = e^{-rt}\l( \begin{matrix} 1 & t & \frac{t^2}{2} & \cdots & \frac{t^p}{p!} \\
                                        0  &\ddots & \ddots &\ddots &\vdots \\
                                        \vdots & \ddots &  \ddots & \ddots & \frac{t^2}{2} \\
                                        \vdots & 0& \ddots & \ddots & t \\
                                        0 & \cdots & \cdots & 0 & 1\end{matrix}\r)
\eea 
so that taking the matrix $A^P$ whose coefficients satisfy $A_{ij}^P = \mathbb{1}_{\{i =1\}} j!a_j$ for $1\leq i,j \leq p$ yields the representation
\beas 
f_P(s) = \langle A^P | e^{-s(D+M)}\rangle.
\eeas 
Moreover, the function 
\beas 
f_C(s) = (c\textnormal{cos}( \xi s) + d \textnormal{sin}(\xi s))e^{-rs}
\eeas 
where $c,d,\xi \in \reels$ can also be represented as before. Defining 
\beas
R = \l(\begin{matrix} 0 &\xi\\-\xi&0\end{matrix}\r),
\eeas 
we verify again that 
\bea \label{cosExample}
\textnormal{exp}(-t(D+R))= e^{-rt}  \l(\begin{matrix} \textnormal{cos}(\xi t) &-\textnormal{sin}(\xi t)\\\textnormal{sin}(\xi t)&\textnormal{cos}(\xi t)\end{matrix}\r),
\eea 
so that if we define the matrix $A^C$ such that $A_{11}^C = c$, $A_{12}^C = -d$, $A_{21}^C=A_{12}^C=0$, then again 
\beas 
f_T(s) = \langle A^T | e^{-s(D+R)}\rangle.
\eeas
More generally we have the following result, which is an easy consequence of the Jordan canonical form for real matrices.
\begin{proposition*}\label{propKernel}
Let $f : \reels_+ \to 
\reels_+$ be a function having the representation 
\beas 
f(s) = \langle A | e^{-sB}\rangle
\eeas 
where $B$ has eigenvalues with positive real parts. Then $f$ is a linear combination of functions of the form 
\bea \label{canonicalFormKernel}
u(s) = P(s) (1 + c\textnormal{cos}(\xi s) + d \textnormal{sin}(\xi s))e^{-rs}
\eea 
where $P(s) = \sum_{k=1}^p a_k s^k$, $\xi, d, c \geq 0$ and $r >0$. Conversely, there exist two matrices $A_u$ and $B_u$ such that $u = \langle A_u | e^{-\cdot B_u}\rangle$, and where $B_u$ has eigenvalues with positive real parts. 
\end{proposition*}
The proof and the forms of the matrices $A_u$ and $B_u$ are relegated to the Appendix. Proposition \ref{canonicalFormKernel} and (\ref{polynomExample})-(\ref{cosExample}) show that the GEMHP encompasses the exponential kernel of \cite{OakesMarkovianHawkes1975}, but also polynomial-exponential kernels (see \cite{ditlevsen2005parameter,pasha2013hawkes,duarte2019stability} and \cite{ditlevsen2017multi}) among others. 

\smallskip

In what follows, our aim is to take advantage of the matrix-exponential shape of the kernel and show that there exists a Markovian process that drives the dynamics of the related marked Hawkes process, provided that the mark process admits a transition kernel given by (\ref{defKernelY}) below. Building on the pure exponential Hawkes process case (see e.g \cite{clinet2017statistical}, p. 1819), we introduce the generalized elementary excitation process  $\cale = (\cale_{\alpha\beta})_{\alpha,\beta=1,...,d}$, defined as 
\bea \label{defExcitationElementaire}
\cale_{\alpha\beta}(t) = \int_{[0,t] \times \mathbb{X}} e^{-(t-s)B_{\alpha\beta}}g_{\alpha\beta}(x)\overline{N}_\beta(ds,dx)  \in \reels^{p \times p}, \alpha,\beta  \in \{1,...,d\},
\eea
for any $t \geq 0$, and write $\E = \reels^{p^2d^2}$ the state space of $\cale$. Naturally, there is no reason to believe that $\cale$ is Markovian when marks affect the individual excitation levels. Accordingly, we turn our attention to the extended process $Z = (\cale,X)$. Proving the Markovian structure of $Z$ requires specific conditions on the mark process that we now detail. First, we introduce $( \kappa_i)_{i \geq 0}$ the sequence of labels of the jumps of the global process $\sum_{\alpha=1}^d N^\alpha$, that is, for $i \geq 1$, $\kappa_i \in \{1,...,d\}$ is the unique (random) index such that $\Delta N_{T_i}^{\kappa_i} =1$, and $\kappa_0$ is a $\{1,...d\}$-valued $\calf_0$-measurable random variable. Now, we assume that there exists a family of Feller transition kernels $\l(\calq_\alpha\r)_{\alpha \in \{1,...,d\}}$ on $\mathbb{X} \times \calx $ such that for $A \in \calx$, $i \geq 1$
\bea \label{defKernelY}
\proba[X_i \in A | \kappa_i, \Delta T_i, \calf_{T_{i-1}}] = \calq_{\kappa_i}(X_{i-1}, A)
\eea 
where, for $i \geq 1$, $\Delta T_i = T_i - T_{i-1}$ with the convention $T_0 = 0$, and where the initial mark $X_0$ is an $\calf_0$-measurable random variable. Hereafter, for $z =(\epsilon,x) \in \E \times \mathbb{X}$, $\alpha \in \{1,...,d\}$, we define 
\bea 
\mu^\alpha(t,z) = \phi_\alpha\l( \l(\langle A_{\alpha\beta} | e^{- t B_{\alpha\beta}}\epsilon_{\alpha\beta} \rangle \r)_{\beta=1,...,d} ,x\r),
\eea 
and $\mu(t,z) = \sum_{\alpha=1}^d \mu^\alpha(t,z)$. Finally, we will use the notations $f^{\alpha}(t,z) = \mu^\alpha(t,z)e^{-\int_0^t \mu^\alpha(s,z)ds}$ and  $f(t,z) = \mu(t,z)e^{-\int_0^t \mu(s,z)ds}$. The next proposition shows that if the mark sequence follows (\ref{defKernelY}), then $Z$ is a Markov process.

\begin{proposition*} \label{propGen}
Assume \textnormal{\textbf{[E]}}. Let $i \geq 1$. We have the following results.
\begin{enumerate}
    \item Given $\calf_{T_{i-1}}$, the law of $\Delta T_i$ depends on $Z_{T_{i-1}}$ only and admits the conditional density $f(\cdot,Z_{T_{i-1}})$ with respect to Lebesgue measure.
    \item We have $\proba[\kappa_i = \alpha | \calf_{T_{i-1}}, \Delta T_i] = \frac{\mu^\alpha(\Delta T_i,Z_{T_{i-1}})}{\mu(\Delta T_i,Z_{T_{i-1}})}$ for $\alpha \in \{1,...,d\}$.
    \item Given $\calf_{T_{i-1}}$ and $\Delta T_i$, the conditional distribution of $X_i$ is given by the measure $$\calq(Z_{T_{i-1}}, \cdot) = \frac{1}{\mu(\Delta T_i,Z_{T_{i-1}})}\sum_{\alpha=1}^{d}\mu^\alpha( \Delta T_i,Z_{T_{i-1}})\calq_{\alpha}(X_{i-1}, \cdot).$$
\end{enumerate}
Consequently, $Z$ is a continuous-time Feller Markov process. Moreover, denoting by $\call$ its associated generator, we have for $z =(\epsilon,x) \in \E \times \mathbb{X}$, $f$ a function that is $C^1$ in its first argument and belonging to the domain of $\call$,
\beas  
\call f(z) = \sum_{\beta=1}^d \int_\mathbb{X} \{f((\epsilon_{\alpha\gamma} + g_{\alpha\gamma}(y)\mathbb{1}_{\{\gamma = \beta\}})_{\alpha,\gamma = 1,...,d}, y) - f(z)\} \psi_\beta(z) \calq_\beta(x,dy) -\sum_{\alpha,\beta=1}^d \frac{\partial f}{\partial \epsilon_{\alpha\beta}}(z). B_{\alpha\beta} \epsilon_{\alpha\beta}, 
\eeas
where $\psi_\alpha(z) = \mu^\alpha(0,z) = \phi_\alpha( (\langle A_{\alpha\beta} | \epsilon_{\alpha\beta} \rangle)_{\beta=1,...,d}, x) $.
\end{proposition*}

The generator $\call$ generalizes simpler cases such as the the queue-reactive model of (\cite{RosenbaumQueue2014}, operator $\calq$, p.109), pure exponential Hawkes process (\cite{clinet2017statistical}, Proof of Proposition 4.5) or the univariate Erlang kernel case (\cite{duarte2019stability}, (3.1)).

\begin{remark*} \label{rmkLinearlyIndependent} We can always assume without loss of generality that for a given $\beta \in \{1,...,d\}$, the components of $e^{-tB_{1\beta}},e^{-tB_{2\beta}}...,e^{-tB_{d\beta}}$ seen as real functions are linearly independent. More precisely, if it were not the case, we could reduce the dimension of $\cale$ by removing subcomponents of the matrices $e^{-tB_{1\beta}},e^{-tB_{2\beta}}...,e^{-tB_{d\beta}}$ until the remaining ones are linearly independent. Moreover, by doing so, we do not break the Markovian structure of the process $Z$ since the removed components are precisely the ones that can be linearly reconstructed from the reduced process. This phenomenon was already pointed out in \cite{clinet2017statistical}, see Remark 3 p.1833.
\end{remark*}

Let us comment on the practical applications of the above proposition. It is important to note that Proposition \ref{propGen} gives an explicit formulation for the conditional distribution of $(\kappa_i, \Delta T_i, X_i)$ for any $i \geq 0$ given the past. This has several important consequences. First, it implies the existence of $(\kappa_i, \Delta T_i, X_i)_{i \geq 0}$ seen as a process, by a straightforward application of Kolmogorov's theorem. In turn, this shows the existence of $Z$ (and hence the random measure $\overline{N}$) which is a deterministic function of the above triplets. Next, $\overline{N}$ is well-formed as a point process with no simultaneous jumps, that is $\proba[\Delta T_i = 0] = 0$ for any $i\geq 0$, since $\Delta T_i$ admits a density with respect to Lebesgue measure by construction. Also, the process never stops ($\proba[\Delta T_i = +\infty] = 0$) as soon as we assume $\phi_\alpha$ bounded from below by some $\eta >0$, for at least one $\alpha \in \{1,...,d\}$. However, further assumptions need to be made to ensure the non-explosiveness of $\overline{N}$ (i.e that $T_n \to +\infty$ as $ \to +\infty$) which will be a consequence of Theorem \ref{thmVgeom} below. Finally, the above construction gives an explicit and iterative way of simulating the process through the distribution of each triplet $(\kappa_i, \Delta T_i, X_i)$, which is of great interest from a practical point of view.  \\

\medskip

We end this section by briefly illustrating our model with a recent popular example from the financial literature (See \cite{RosenbaumQueue2014} for the original queue-reactive model, and more recently \cite{wu2019queue} for a partial extension featuring a Hawkes component). In its simplest form, a queue-reactive Hawkes process consists of a 3-dimensional counting process $(M,L,C)$ whose stochastic intensity and excitation kernel admit the form (\ref{generalIntensityhab})-(\ref{defHgenExp}), representing three order flows happening at a given price level for a particular stock on an electronic market. Market (M) orders (i.e trades) and cancellation (C) orders are events that decrease by one unit the size of the queue $X = X_0 + L - M - C$ at this price limit, whereas limit (L) orders correspond to new incoming orders that increase the size of $X$. In this model, the queue $X$ itself is the mark process, and the state-space is $\mathbb{X} = \naturels$ endowed with the discrete distance inherited from the absolute value $|\cdot|$. In addition to the above specification, we require that the transition kernels for the marks are (with obvious notations) $\calq_L(x,\cdot) = \delta_{x+1}(\cdot)$, $\calq_M(x,\cdot) =\calq_C(x,\cdot) = \delta_{0}(\cdot) \mathbb{1}_{\{x = 0 \}} + \delta_{x-1}(\cdot) \mathbb{1}_{\{x \neq 0 \}}$, where $\delta_y$ is the Dirac measure at point $y$. For simplicity, cancellations are assumed non-exciting, nor can they be excited by other components. For non-degeneracy reasons that will become apparent later, we also assume that when the limit is empty, market and cancellation orders may still occur with an arbitrary rate $\underline{\phi} > 0$. Note that since $0$ is an absorbing state for $\calq_C$ and $\calq_M$, this does not affect the distribution of the queue process in any way. A typical and simple specification for the intensity functions is therefore
\bea
\phi_L(x,u) &=& \nu_L + \sum_{\beta \in \{L,M,C\}} u_\beta, \label{phiL}\\
\phi_M(x,u) &=& \nu_M \mathbb{1}_{\{ x  \neq 0\}} + \underline{\phi}\mathbb{1}_{\{ x  = 0\}} + \sum_{\beta \in \{L,M,C\}} u_\beta, \label{phiM}\\
\phi_C(x,u) &=& \nu_C x + \underline{\phi}\mathbb{1}_{\{ x  = 0\}}, \label{phiC}
\eea 
where $\nu_L,\nu_M, \nu_C >0$. The rationale behind the linearity of $\phi_C$ in $x$ is the following: the limit  orders placed in the queue can be independently cancelled with a Poisson intensity $\nu_C$ at any time (see \cite{abergel2015long}). The above model can be complexified to meet more general conditions, as, for instance, dependence of the other intensity functions in $x$, or a multidimensional queuing system $X$ (see, e.g \cite{RosenbaumQueue2014}). Moreover, random sizes of orders also yield more complex transition kernels $(\calq_L, \calq_M,\calq_C)$.   

\subsection*{$V$-geometric ergodicity of the GEMHP}

In this section, we establish under suitable conditions the $V$-geometric ergodicity of the multidimensional GEMHP. Our strategy consists in first finding a Lyapunov function for the process $Z$, and then in proving that the transition kernel of $Z$ satisfies a non-degeneracy condition given in Lemma \ref{lemND}. We start with a few stability conditions that ensure the existence of a Lyapunov function. We show that those conditions are met in the case of the queue-reactive Hawkes process for the sake of illustration.

\bd
\im[[L1\!\!]] There exists a continuous function $\nu : \mathbb{X} \to \reels_+^d$, such that for all $\alpha =1,...,d$, $\phi_{\alpha}(x,u) \leq \nu_\alpha(x) + \sum_{ \beta = 1}^d u_\beta  $, $x \in \mathbb{X}, u \in \reels_+^d$. 
\ed
The above assumption states that the stochastic intensity should be sub-linear in the excitation terms. Although recent attention has been devoted to quadratic Hawkes processes (see for instance \cite{blanc2017quadratic}), we set aside these processes in the present work. Note that \textbf{[L1]} can be rewritten in terms of $\psi$, as  
$\psi_{\alpha}(z) \leq \nu_\alpha(x) + \sum_{ \beta = 1}^d \langle A_{\alpha\beta}| \epsilon_{\alpha\beta} \rangle $. In the next assumption we call norm-like function any non-negative function $f$ such that $f(x) \to +\infty$ as $|x| \to +\infty$. Moreover we define for $x \in \mathbb{X}$ the quantity $G_{\alpha\beta}(x) = \int_\mathbb{X} g_{\alpha\beta}(y)\calq_\beta(x,dy)$, corresponding to the conditional expected size of a jump in $\lambda^\alpha$ when $N^\beta$ jumps.
\bd
\im[[L2\!\!]] There exist two norm-like functions $f_X, u_X$, and a constant, $L \geq 0$, such that for $|x| \to +\infty$, $\sum_{\alpha =1}^d \nu_{\alpha}(x) = O(u_X(x))$,  and for any $x\in \mathbb{X}$
\bea \label{driftConditionMark}
\sum_{\alpha =1}^d \nu_{\alpha}(x)\int_\mathbb{X}  [f_X(y) - f_X(x)] \calq_\alpha(x, dy) \leq - u_X(x). 
\eea 
Moreover, for $|x| \to +\infty $
\bea \label{driftConditionMark2}
\sum_{\alpha =1}^d \nu_{\alpha}(x) \sum_{\beta = 1}^d G_{\beta\alpha}(x) = o(u_X(x)).
\eea 
Finally, we assume that there exists $\overline{c} > 0$ such that 
\bea \label{momentDriftExpo}
\sup_{x \in \mathbb{X}, \beta \in \{1,...,d\} } \int_\mathbb{X}   e^{\overline{c}\l[ \sum_{\alpha = 1}^d g_{\alpha\beta}(y) + f_X(y) - f_X(x) \r]} \calq_\beta(x, dy) < +\infty.
\eea 
\ed
Assumption \textbf{[L2]} essentially assumes the existence of a drift condition for the mark process itself. More precisely, (\ref{driftConditionMark}) states that, when properly weighted by the baselines $\nu_\alpha$, there exists a Lyapunov function $f_X$ for $X$, with rate $u_X$ of order at least $\sum_{\alpha=1}^d \nu_\alpha$. Next, (\ref{driftConditionMark2}) represents the contribution of the baseline terms $\nu_\alpha(x)$ in the stochastic intensities to the short-term variation of the excitation levels. It implies that as $|x|$ increases, although the combined effect may tend to infinity, it should be negligible with respect to the rate of the drift-condition (\ref{driftConditionMark}). Finally (\ref{momentDriftExpo}) is a uniform moment condition on the boost functions and $f_X$. Let us specify briefly what happens for the queue-reactive model introduced in (\ref{phiL})-(\ref{phiC}). Clearly \textbf{[L1]} holds with equality. Next, the left-hand side of (\ref{driftConditionMark}) reduces to $\nu_L - \nu_M - \nu_Cx $, while in (\ref{driftConditionMark2}) it is bounded because cancellation orders do not trigger any excitation and $G_{\beta\alpha}$ is bounded for all $\alpha,\beta \in \{L,M,C\}$ by (\ref{momentDriftExpo}). Therefore \textbf{[L2]} is satisfied with $u_X(x) = -\nu_L + \nu_M + \mu_Cx $. Here, the cancellation term $\nu_C x$ plays a crucial role as it pushes the size of the limit back when $X$ becomes too large. Finally, we give a stability condition on the excitation kernels. We let, for $x \in \mathbb{X}$, $s \geq 0$
\bea
H(s,x) = \l( \int_\mathbb{X} h_{\alpha\beta}(s,y)\calq_\beta(x,dy) \r)_{\alpha,\beta = 1,...,d} \in \reels_+^{d \times d},
\eea
and
\bea 
\Phi(x) = \int_0^{+\infty} H(s,x)ds \in\reels_+^{d \times d}.
\eea 
Since all the $B_{\alpha\beta}$'s are invertible with eigenvalues having negative real parts, we have the representation
\bea \label{repPhi}
\Phi_{\alpha\beta}(x) =  \langle A_{\alpha\beta} | B_{\alpha\beta}^{-1} \rangle G_{\alpha\beta}(x)\textnormal{, } \alpha,\beta \in \{1,...,d\}.
\eea 
$\Phi$ corresponds to the conditional expectation of the long-run contribution of the excitation functions to the stochastic intensity after a jump. It is a generalized version of the standard excitation matrix for unmarked Hawkes process, introduced in \cite{BremaudStability1996}, Theorem 7. 
\bd
\im[[L3\!\!]] There exists $\kappa \in \reels^d$ with positive coefficients, and $\rho \in [0,1)$ such that, component-wise, $$\sup_{x \in \mathbb{X}} [\Phi(x)^T \kappa] \leq \rho \kappa.$$ 
\ed 
There are several situations where assumption \textbf{[L3]} can be greatly simplified. First, if the marks are independent and identically distributed, then $\Phi$ is independent of $x$. In this case, it is easy to see that by the Perron-Frobenius theorem, \textbf{[L3]} is satisfied if $\rho(\Phi)$, the spectral radius of $\Phi$, is smaller than unity. This corresponds exactly to the standard stability condition derived in \cite{BremaudStability1996} for general multivariate Hawkes processes. When the marks are not independent, but the point process is univariate, \textbf{[L3]} reduces to $\sup_{x \in \mathbb{X}} \int_0^{+\infty}h(s,x)ds  < 1$. In other words, $h$ should have an $\mathbb{L}_1$ norm uniformly smaller than $1$ in time. Finally note that Condition \textbf{[L3]} is the main obstacle against including negative functions $h_{\alpha\beta}$. The Lyapunov function we use below seems difficult to adapt when allowing the matrix-exponential kernels to be negative which is why they have been set aside in this paper. In particular, there is no clear way of generalizing the excitation matrix (\ref{repPhi}) so that it can then be exploited with a correct drift function. We are now ready to derive a Lyapunov function for $Z$.

\begin{lemma*} \label{lemDrift}
Assume \textnormal{\textbf{[E]}} and \textnormal{\textbf{[L1]-[L3]}}. For any $\alpha,\beta \in \{1,...,d\}$ there exist a matrix $a_{\alpha\beta} \in \reels^{p \times p}$ with positive coefficients and $\eta >0$ such that, for $z \in \E \times \mathbb{X}$, defining $V(z) = \textnormal{exp}(\sum_{\alpha, \beta =1}^d a_{\alpha\beta} \epsilon_{\alpha\beta} + \eta f_X(x) )$, we have the drift condition
\beas 
\call V \leq -\delta V + L,
\eeas
for some $\delta>0$, $L \geq 0$.
\end{lemma*}

For $T \geq 0$, we define $P^T$ the transition kernel for $Z$. We now give two conditions that ensure the non-degeneracy, and the existence of an open accessible small set for $P^T$.

\bd
\im[[ND1\!\!]] There exist $\underline{\phi}, \underline{g} > 0$ such that for any $\alpha,\beta \in \{1,...,d\}$, $\phi_{\alpha} \geq \underline{\phi}$ and $g_{\alpha\beta} \geq \underline{g}$.
\ed

We recall that a point $x \in \mathbb{X}$ is reachable for a transition kernel if any neighborhood $V$ of $x$ is accessible for this kernel.

\bd
 \im[[ND2\!\!]] The transition kernel $\calq$ admits a reachable point $x_0 \in \mathbb{X}$. Moreover, for every $\beta \in \{1,...,d\}$, the transition kernel $\calq_\beta$ admits a sub-component $\calt_\beta$, such that there exists a non-trivial measure $\sigma_\beta$ on $\mathcal{X}$ and a lower semi-continuous non-negative function $r_\beta : \mathbb{X}^2 \to \reels_+$, such that 
 \begin{itemize}
     \item For any non-empty open set $O \in \mathcal{X}$, $\sigma_\beta(O) >0.$
     \item For any $x \in \mathbb{X}$, $A \in \mathcal{X}$, $\calt_\beta(x,A) = \int_\mathcal{X} r_\beta(x,y) \sigma_\beta(dy)$, and $\calt_\beta(x,\mathbb{X}) > 0$.
 \end{itemize}

\ed 

\begin{remark*}
Note that by Fatou's lemma and the semi-lower continuity of $r$ in its first argument, $\calq_{\beta}$ is a $T$-kernel for any $\beta \in \{1,...,d\}$ (\cite{douc2018markov}, Definition 12.2.1, p. 270). If $\mathbb{X}$ is discrete (as in the queue-reactive Hawkes case) then \textnormal{\textbf{[ND2]}} is trivially satisfied (taking $\sigma$ as the counting measure and the discrete topology). Note also that for the queue-reactive Hawkes model, $0$ is a reachable point since a trade or a cancellation will always occur with non-zero probability when the limit is non-empty. Finally, if $\mathbb{X} \subset \reels^q$ and $\sigma$ is the Lebesgue measure, then the first point is automatically satisfied.
\end{remark*}

In the next lemma, $\textbf{B}(\E)$ stands for the Borel $\sigma$-field associated to $\E = \reels^{p^2d^2}$.

\begin{lemma*} \label{lemND}
There exists $T > 0$ such that $P^T$ admits an open accessible small set $U \in \textbf{B}(\E) \otimes \mathcal{X} $, that is, there exists a non-trivial measure $\underline{\nu}$ such that 
\beas 
\inf_{z \in U} P^T(z,dz') \geq \underline{\nu}(dz').
\eeas
Consequently, all compacts sets are petite.  
\end{lemma*}


We now use the results of the two previous paragraphs and conclude this section by establishing the $V$-geometric ergodicity of the Markov process $Z$. Recall that for a positive function $V$, the $V$-norm of a measure $\mu$ on a measurable space $(S,\cals)$ is defined as 
\beas 
\|\mu \|_{V}  = \sup_{ \psi| \psi \leq V}\l| \int_S \psi(s) \mu(ds)\r|
\eeas 
where the supremum is taken over all the measurable functions $\psi : S \to \reels_+$ such that $\psi \leq V$.

\begin{theorem*} \label{thmVgeom}
($V$-geometric ergodicity) Assume \textnormal{\textbf{[E]}}, \textnormal{\textbf{[L1]-[L3]}} and \textnormal{\textbf{[ND1]-[ND2]}}. Then $Z$ is $V$-geometric ergodic: there exists a unique invariant measure $\pi$, two constants $C \geq 0$, $0 \leq r < 1 $  such that for any $T >0$, for any $z \in \E \times \mathbb{X}$
$$ \| P^T(z, .) - \pi\|_V \leq C(1+V(z))r^T $$
where $V(z) = \textnormal{exp}(\sum_{\alpha, \beta =1}^d a_{\alpha\beta} \epsilon_{\alpha\beta} + \eta f_X(x) )$, with $(a_{\alpha\beta})_{\alpha,\beta \in \{1,...,d\}}$, $\eta$ and $f_X$ as in Lemma \ref{lemDrift}. Moreover, $Z$ is $V$-geometrically mixing, that is there exists $C' >0$, $0 \leq r' < 1 $ such that for any $t,u \geq 0$, for any $\phi,\psi$ such that  $\phi^2 \leq V$, $\psi^2 \leq V$, there
\beas 
\l|\esp[\phi(Z_{t+u})\psi(Z_{t}) | Z_0 = z] - \esp[\phi(Z_{t+u}) | Z_0 = z]\esp[\psi(Z_{t}) | Z_0 = z] \r| \leq C' V(z) r'^u.
\eeas 

\end{theorem*}

Theorem \ref{thmVgeom} is closely related to several results from the existing literature. In the Hawkes process literature, \cite{abergel2015long} first derived a polynomial Lyapunov function for the restricted case of a Hawkes process having pure exponential kernels. Proposition 4.5 from \cite{clinet2017statistical} extended the argument to the case of an exponential drift function $V(\epsilon) = \textnormal{exp}(\sum_{\alpha, \beta =1}^d a_{\alpha\beta} \epsilon_{\alpha\beta})$, and establishes the full $V$-geometric ergodicity of the process. In \cite{duarte2019stability}, the authors construct a Lyapunov function (Proposition 3) and then establish the exponential convergence towards $0$ of the Wasserstein distance between $\nu P^T$ and $\pi$ for any starting measure $\nu$ (Theorem 1), in the case where the original process $\overline{N}$ is unmarked and univariate, and where the excitation kernel $h$ admits the form $h(s) = \sum_{i=1}^P c_i s^i e^{-r_i s}$, $r_i >0$ for $i=1,...,P$. An important consequence of Theorem \ref{thmVgeom} is the following result.
 
\begin{corollary*} \label{corolVgeom}
 Assume \textnormal{\textbf{[E]}}, \textnormal{\textbf{[L1]-[L3]}} and \textnormal{\textbf{[ND1]-[ND2]}}. Then, up to a change of probability space, there exists a two-sided stationary version $\overline{N}' = (\overline{N}_\alpha')_{\alpha \in \{1,...,d\}}$ of the GEMHP, where $\overline{N}'$ is a family of random measures on $\reels \times \mathbb{X}$ and where its stochastic intensity $\lambda'$ satisfies for any $t \in \reels$
 \bea \label{eqInfinity}
 \lambda'^\alpha(t) = \phi_\alpha \l( \l(\int_{(-\infty,t) \times \mathbb{X}} \langle A_{\alpha\beta} | e^{-(t-s) B_{\alpha\beta}} \rangle g_{\alpha\beta}(x)\overline{N}_\beta'(ds,dx)\r)_{\beta =1,...,d},X_{t-}'\r),
 \eea 
where $X'$ is the associated stationary two-sided mark process.
\end{corollary*}

\section{Quasi-likelihood analysis for the GEMHP} \label{sectionQLAHawkes}

We finally use Theorem \ref{thmVgeom} to prove that, up to some moment and identifiability conditions, the Quasi-Likelihood Analysis of Section \ref{sectionQLA} applies to the GEMHP, which is stated in Corollary \ref{thmFinal}. For the sake of tractability, we restrict ourselves to the case of a linear process, that is when $\phi_\alpha(x,u) = \nu_\alpha(x) + \sum_{\beta = 1}^d u_\beta$ for any $\alpha \in \{1,...,d\}$. The case where $\phi^\alpha$ is sub-linear and bounded from below uniformly in $\theta$ (among other things) would involve much more involved formulations for Assumpions \textbf{[AH1]-[AH2]} below, which is why it is let aside for future research. Accordingly, we assume that for each parameter $\theta \in \Theta$ and any $\alpha \in \{1,...,d\}$ we have
\bea \label{lambdaGEMHPQLA}
\lambda^\alpha(t, \theta ) =  \nu_\alpha(X_{t-}, \theta) + \sum_{\beta=1}^d\int_{[0,t) \times \mathbb{X}} h_{\alpha\beta}(t-s, x, \theta) \overline{N}_\beta(ds, dx), 
\eea
where, $h_{\alpha\beta}(s,x,\theta) = \langle A_{\alpha\beta}(\theta) | e^{-s B_{\alpha\beta}(\theta)}\rangle g_{\alpha\beta}(x,\theta)$ and such that the real parts of the eigenvalues of $B_{\alpha\beta}(\theta)$ are all bounded from below by some $r  > 0$ which is independent of $\theta \in \Theta$. We do not explicitly specify the shape of each component in (\ref{lambdaGEMHPQLA}) as a function of $\theta$, although, for $h_{\alpha\beta}$, the canonical form (\ref{canonicalFormKernel}) obtained in Proposition \ref{propKernel} yields a very natural parametrization for the excitation kernel through the coefficients appearing in (\ref{canonicalFormKernel}). We now turn our attention to the mark transition kernel. We assume that there exists a dominating measure $\rho$ on $\mathcal{X}$ which plays the same role as in Section \ref{sectionQLA} and a density $p_\beta$ that are such that for any $\beta \in \{1,...,d\}$, and $\theta \in \Theta$
\bea \label{densityQbeta} 
\calq_{\beta}(x,dy,\theta) = p_{\beta}(x,y,\theta) \rho(dy).
\eea 
Note that then, using the notations of Section \ref{sectionQLA}, we obtain that the mark density $q^\alpha$ is simply 
\bea \label{repqparam}  
q^\alpha(t,x,\theta) =  p_\alpha(X_{t-}, x,\theta)
\eea
as an immediate consequence of (\ref{densityQbeta}). In what follows, we assume that there exists $\theta^* \in \Theta$ corresponding to the true parameter of the model, and hereafter, we will always assume that the underlying GEMHP associated to $\theta^*$ satisfies assumptions \textbf{[L1]-[L3]} and \textbf{[ND1]-[ND2]} so that Theorem \ref{thmVgeom} and Corollary \ref{corolVgeom} hold. Parts of Assumptions \textbf{[A1]-[A4]} are naturally reformulated below for the GEMHP.

\bd
\im[[AH1\!\!]] For any $\alpha,\beta \in \{1,...,d\}$
  \bd
		\im[{\bf (i)}] $\nu_\alpha(x,.),A_{\alpha\beta}, B_{\alpha\beta}, g_{\alpha\beta}(x,\cdot), p_\alpha(u,x,\cdot)$ are in $C^4(\Theta)$ and admit continuous extensions on $\overline{\Theta}$.
		\im[{\bf (ii)}] For any $\theta \in \Theta$, $\lambda^\alpha(t,\theta)q^\alpha(t,x,\theta) = 0 $ if and only if $\lambda^\alpha(t,\theta^*)q^\alpha(t,x, \theta^{*}) =0, dt \rho(dx)$-a.e.
	\ed
\ed

\bd
\im[[AH2\!\!]] For any $p > 1$, for any $\alpha,\beta \in \{1,...,d\}$, uniformly in $\theta \in \Theta$ we have
        \bd 
        \im[{\bf (i)}]$ \sum_{i=0}^3    \l[|\partial_\theta^i\nu_\alpha(x,\theta)|^p + |\nu_\alpha(x,\theta)\mathbb{1}_{\{\nu_\alpha(x,\theta) \neq 0\}}|^{-p}+ \l|\partial_\theta^i g_{\alpha\beta}(x, \theta)\r|^p + \r]   \leq C_pe^{\eta f_X(x)}, $
        \im[{\bf (ii)}] $  \sum_{i=0}^3  \int_{\mathbb{X}} |\partial_\theta^i \textnormal{log}p_\alpha(x,y,\theta)|^p p_{\alpha}(x,y,\theta^*) \rho(dy)  \leq C_pe^{\eta f_X(x)},$
        \ed
        where $C_p \geq 0$ may depend on $p$, and $f_X$ is defined in \textbf{[L2]}.
	
\ed

A first consequence of the $V$-geometric ergodicity of the GEMHP and $\textbf{[AH1]-[AH2]}$ is the following shape for the limit field $\mathbb{Y}$.

\begin{lemma*} \label{lemLimitYHawkes}
For any $\theta \in \Theta$, we have 
$$ \mathbb{Y}(\theta) = -\sum_{\alpha=1}^d \esp\l[\int_{\mathbb{X}}\textnormal{log}\l(\frac{f'^\alpha(0,x,\theta)}{f'^\alpha(0,x,\theta^*)}\r)f'^\alpha(0,x,\theta^*) - (f'^\alpha(0,x,\theta)-f'^\alpha(0,x,\theta^*))\rho(dx)\r],$$
where $f'^\alpha(0,x,\theta) = \lambda'^\alpha(0,\theta)q'^\alpha(0,x,\theta)$, corresponds to the density of the predictable compensator of the stationary version $\overline{N}'$ from Corollary \ref{corolVgeom} at time $0$.
\end{lemma*}

Our final assumption before stating our main result is the direct translation of the identifiability condition $\textbf{[A4]}$, which may be checked using the formula derived in Lemma \ref{lemLimitYHawkes}, depending on the particular parametrization of the different components appearing in (\ref{lambdaGEMHPQLA}).
\bd
\im[[AH3\!\!]] We have $\textnormal{inf}_{\theta \in \Theta - \{\theta^*\}} -\frac{\mathbb{Y}(\theta)}{|\theta - \theta^*|^2} > 0.$
	
\ed

\begin{corollary*} \label{thmFinal}
Assume  \textnormal{\textbf{[AH1]-[AH3]}}, along with conditions \textnormal{\textbf{[L1]-[L3]}} and \textnormal{\textbf{[ND1]-[ND2]}}. Then \textnormal{\textbf{[A1]-[A4]}} are satisfied for any $\gamma \in (0,1/2)$, and in particular we have for the GEMHP that for any continuous function $u$ of polynomial growth, 
$$ \esp[u(\sqrt T (\hat{\theta}_T - \theta^*))] \to \esp[u(\Gamma^{-1/2} \xi)],$$
$$ \esp[u(\sqrt T (\tilde{\theta}_T - \theta^*))] \to \esp[u(\Gamma^{-1/2} \xi)]$$
as $T \to + \infty$, where $\xi$ follows a standard normal distribution, $\Gamma$ is the Fisher information matrix and $\hat{\theta}_T$ and $\tilde{\theta}_T$ are respectively the QMLE and any QBE of the model.
\end{corollary*}

\section{Conclusion}

We have introduced a general parametric framework for multivariate marked point processes observed on the real half line, and given general ergodicity assumptions so that the QMLE and the QBE enjoy asymptotic normality along with convergence of their moments. We have then shown, as a main application, that marked Hawkes processes having generalized exponential kernels in time and marks satisfying among others a Lyapunov condition yield a $V$-geometrically ergodic Markovian system and hence fall under the scope of our statistical framework. Finally, we have illustrated our main assumptions on the marked Hawkes process with the simple case of the queue-reactive Hawkes model. \\

There are some points left to explore, such as what happens to the marked Hawkes process with a more general, non-Markovian kernel, or even non-Markovian marks. Although it would shed more light on the applicability of the QLA method, it would imply a radical change in the way the ergodicity condition (with rate $T^\gamma$ for some $\gamma \in (1/2)$) is proved since the present paper heavily relies on the Markovian representation of the marked point process. We do not pursue further this investigation, which is left for future research.  

\section*{Acknowledgments}
The research of Simon Clinet is supported by Japanese Society for the Promotion of Science Grant-in-Aid for Young Scientists No. 19K13671.


\section{Appendix: Proofs and Technical Results}

\subsection{Proofs of Section \ref{sectionQLA}}
\begin{lemma*} \label{lemBDG}
Let $W $ be a predictable function on $\Omega \times \reels_+ \times \mathbb{X}$. Then for any $\alpha \in \{1,...,d\}$, $p \geq 1$, we have
\beas 
\esp\l| \int_{[0,T] \times \mathbb{X}} W(s,x) \tilde{M}^\alpha(ds,dx)\r|^{2^p} &\leq& C_p \esp \int_{[0,T] \times \mathbb{X}} |W(s,x)|^{2^p} f(s,x,\theta^*) ds \rho(dx) \\&+& C_p \esp \l|\int_{[0,T] \times \mathbb{X}} W(s,x)^2 f(s,x,\theta^*) ds\rho(dx) \r|^{2^{p-1}}, 
\eeas 
where $C_p \geq 0$ depends on $p$ only, whenever the expectations are well defined.
\end{lemma*}

\begin{proof}
This follows from the proof of Lemma A.2 from \cite{clinet2017statistical}, replacing "$f_s$" by "$W(s,x)$", "$\tilde{N}^\alpha$" by "$\tilde{M}^\alpha(ds,dx)$", and "$\lambda(s,\theta^*)ds$" by "$f(s,x,\theta^*)ds\rho(dx)$". Moreover, the probability measure $\mu(dt)$ on $[0,T]$ should be changed to 
$$\mu(dt,dx) = \l(\int_{[0,T] \times \mathbb{X}} W(s,x)^2 f(s,x,\theta^*)ds\rho(dx)\r)^{-1} W(t,x)^2 f(t,x,\theta^*)dt\rho(dx),$$ 
on $[0,T] \times \mathbb{X}$.
\end{proof}

\begin{lemma*} \label{lemAppliSobolev} Let $(V_t(\theta))_{t \in \reels_+, \theta \in \Theta}$ be an $\reels^q$-valued random field for some $q \geq 1$ such that
\begin{enumerate} 
\item for all $t \in \reels_+$, $(V_t(\theta))_{\theta \in \Theta}$ is in $C^1(\Theta)$.
\item for some $p > n$, there exist $\mathbb{V}(\theta)$ and $\mathbb{W}(\theta)$ such that $\sup_{\theta \in \Theta} \|V_t(\theta) - \mathbb{V}(\theta)\|_p \to 0$, and $\sup_{\theta \in \Theta} \|\partial_\theta V_t(\theta) - \mathbb{W}(\theta)\|_p \to 0$.
\end{enumerate}
Then, $  \| \sup_{\theta  \in \Theta}|V_t(\theta) - \mathbb{V}(\theta)|\|_p \to 0.$
\end{lemma*}
\begin{proof}

Applying Sobolev's inequality (Theorem 4.12 from \cite{adams2003sobolev} part I, Case A, $j=0$, $m=1$, $p > n$) for $T,T' \in \reels_+$, we have 
\beas 
\esp \sup_{\theta \in \Theta} |V_T(\theta) - V_{T'}(\theta)|^p &\leq&  K(\Theta,p)\l(\int_{\Theta}{d\theta \esp\l|V_T(\theta) - V_{T'}(\theta) \r|^p}+\int_{\Theta}{d\theta \esp\l|\partial_\theta V_T(\theta)-\partial_\theta V_{T'}(\theta)\r|^p} \r)\\
&\leq& K(\Theta,p) \textnormal{diam}(\Theta) \l(\sup_{\theta \in \Theta} \esp |V_T(\theta) - V_{T'}(\theta)|^p + \sup_{\theta \in \Theta} \esp |\partial_\theta V_T(\theta) -  \partial_\theta V_{T'}(\theta)|^p\r),
\eeas 
and this tends to $0$ as $T,T' \to +\infty$, so that $V_T$ (seen as a sequence indexed by $T$) is a Cauchy sequence for the norm $\esp \sup_{\theta \in \Theta} |\cdot|$, and hence converges for this norm to a limit $V$. Of course, we necessarily have $V  = \mathbb{V} $.

\end{proof}

\begin{lemma*} \label{lemBDGmartingale}
Let $(M_t(\theta))_{t \in \reels_+, \theta \in \Theta}$ be an $\reels^q$-valued random field for some $q \geq 1$, and let $p >n$ be such that 
\begin{enumerate}
    \item for any $\theta \in \Theta$, $(M_t(\theta))_{t \in \reels_+}$ is an $\mathbb{L}^p$ integrable martingale.
    \item for any $t \in \reels_+$, $(M_t(\theta))_{\theta \in \Theta}$ is in $C^1(\Theta)$.
\end{enumerate}
Then for any $\theta \in \Theta$, $(\partial_\theta M_t(\theta))_{t \in \reels_+}$ is again an $\mathbb{L}^p$ integrable martingale, and for any $T \in \reels_+$, there exists a positive constant $C(p,\Theta)$ such that
\beas 
\esp \l[\sup_{\theta \in \Theta}|M_T(\Theta)|^p\r] \leq C(p,\Theta)\l( \sup_{\theta \in \Theta} \esp \l[ M_T(\theta), M_T(\theta)\r]^{p/2} +  \sup_{\theta \in \Theta}\esp \l[ \partial_\theta M_T(\theta), \partial_\theta M_T(\theta)\r]^{p/2}\r),
\eeas 
where $[\cdot,\cdot]$ is the quadratic variation operator.
\end{lemma*}

\begin{proof}
The fact that $(\partial_\theta M_t(\theta))_{t \in \reels_+}$ is an $\mathbb{L}^p$ integrable martingale is an easy consequence of the dominated convergence theorem. Applying again Sobolev's inequality (Theorem 4.12 from \cite{adams2003sobolev} part I, Case A, $j=0$, $m=1$, $p > n$) and then Burkholder-Davis-Gundy inequality yields
\beas 
\esp \sup_{\theta \in \Theta} |M_T(\theta) |^p &\leq&  K(\Theta,p)\l(\int_{\Theta}{d\theta \esp\l|M_T(\theta)  \r|^p}+\int_{\Theta}{d\theta \esp\l|\partial_\theta M_T(\theta)\r|^p} \r)\\
&\leq&  K(\Theta,p) \textnormal{diam}(\Theta) \l( \sup_{\theta \in \Theta} \esp \l[ M_T(\theta), M_T(\theta)\r]^{p/2} +  \sup_{\theta \in \Theta}\esp \l[ \partial_\theta M_T(\theta), \partial_\theta M_T(\theta)\r]^{p/2}\r).
\eeas 
\end{proof}

\begin{proof}[Proof of Lemma \ref{lemY}]
By Assumption \textbf{[A1]}-\textbf{[A3]} and Lemma 3.15 from \cite{clinet2017statistical}, there exists $\mathbb{Y}^{(1)}(\theta) $ such that $T^\gamma\|\sup_{\theta \in \Theta} |\mathbb{Y}_T^{(1)}(\theta) - \mathbb{Y}^{(1)}(\theta)| \|_p \to 0$, therefore we only need to prove the existence of $\mathbb{Y}^{(2)}(\theta)$ such that  $T^\gamma \|\sup_{\theta \in \Theta} |\mathbb{Y}_T^{(2)}(\theta) - \mathbb{Y}^{(2)}(\theta)| \|_p \to 0$ for any $p \geq 1$. Let us define 
$$\tilde{\mathbb{Y}}_T^{(2)}(\theta) = T^{-1}\sum_{\alpha=1}^d\int_{[0,T] \times \mathbb{X}} \textnormal{log} \frac{q^\alpha(s,x,\theta)}{q^\alpha(s,x,\theta^*)} q^\alpha(s,x,\theta^*) \lambda^\alpha(s,\theta^*)ds \rho(dx),$$
and prove that 
\bea \label{convY2Ytilde2}
T^\gamma \|\sup_{\theta \in \Theta} |\mathbb{Y}_T^{(2)}(\theta)- \tilde{\mathbb{Y}}_T^{(2)}(\theta)| \|_p \to 0. 
\eea 
Note that 
\bea \label{defYtilde2} 
\mathbb{Y}_T^{(2)}(\theta)- \tilde{\mathbb{Y}}_T^{(2)}(\theta) = T^{-1}\sum_{\alpha=1}^d\int_{[0,T] \times \mathbb{X}} \textnormal{log} \frac{q^\alpha(s,x,\theta)}{q^\alpha(s,x,\theta^*)} \tilde{M}^\alpha(ds,dx).
\eea
Without loss of generality, we can assume that $p = 2^q$ for some $q \geq 1$, and that $p >n = \textnormal{dim}(\Theta)$. Next, by \textbf{[A2]}, we know that $S_T^\alpha = \int_{[0,T] \times \mathbb{X}} \textnormal{log} \frac{q^\alpha(s,x,\theta)}{q^\alpha(s,x,\theta^*)} \tilde{M}^\alpha(ds,dx)$ is an $\mathbb{L}^p$-integrable martingale, so that by Lemma \ref{lemBDGmartingale}, we will have 
				
		        \bea \label{ineqSobolevS}
					\esp \sup_{\theta \in \Theta}\l| \frac{S_T^\alpha(\theta)}{T} \r|^p = o\l(T^{-\gamma p}\r)
				\eea
				as soon as we show $\sup_{\theta \in \Theta} \esp \l[\partial_\theta^i S_T^\alpha(\theta),\partial_\theta^i S_T^\alpha(\theta)\r]^{p/2} = o(T^{-\gamma p})$ for $i \in \{0,1\}$ where $[\cdot,\cdot]$ denotes the quadratic variation operator (the fact that $\partial_\theta S_T^\alpha$  is again an $\mathbb{L}^p-$integrable martingale is an easy consequence of the dominated convergence theorem). Now, by application of Lemma \ref{lemBDG} along with Jensen's inequality with respect to the measure $T^{-1}q^\alpha(s,x,\theta^*)ds\rho(dx)$, the fact that $|\textnormal{log} x| \leq |x| + |x|^{-1} $, and finally assumption \textbf{[A2]}, we have for some constant $C>0$ that may change from one line to the next
				\beas
				\esp \l[ S_T^\alpha(\theta),  S_T^\alpha(\theta)\r]^{p/2} &\leq & C \esp \l(\int_{[0,T] \times \mathbb{X}}{\l|\textnormal{log}\frac{q^\alpha(s,x,\theta)}{q^\alpha(s,x,\theta^*)}\r|^2f^\alpha(s,x,\theta^*)\mathbb{1}_{\{\lambda^\alpha(s,\theta^{*}) \neq 0\}}ds\rho(dx)} \r)^{\frac{p}{2}} \\
				&+& C \esp \int_{[0,T] \times \mathbb{X}} {\l|\textnormal{log}\frac{q^\alpha(s,x,\theta)}{q^\alpha(s,x,\theta^*)}\r|^p f^\alpha(s,x,\theta^*)\mathbb{1}_{\{\lambda^\alpha(s,\theta^{*}) \neq 0\}}  ds\rho(dx)}\\
				&\leq& C T^{\frac{p}{2}-1} \esp \int_{[0,T] \times \mathbb{X}}{\l|\textnormal{log}\frac{q^\alpha(s,x,\theta)}{q^\alpha(s,x,\theta^*)}\r|^p \lambda^\alpha(s,\theta^*) ^{\frac{p}{2}}\mathbb{1}_{\{\lambda^\alpha(s,\theta^{*}) \neq 0\}}q^\alpha(s,x,\theta^*)ds\rho(dx)} \\
				&+& C \esp \int_{[0,T] \times \mathbb{X}} {\l|\textnormal{log}\frac{q^\alpha(s,x,\theta)}{q^\alpha(s,x,\theta^*)}\r|^p f^\alpha(s,x,\theta^*)\mathbb{1}_{\{\lambda^\alpha(s,\theta^{*}) \neq 0\}}  ds\rho(dx)}\\
				&=& O\l(T^{\frac{p}{2}}\r), 
				\eeas

				and 
				\beas 
				\esp \l[\partial_\theta^i S_T^\alpha(\theta),\partial_\theta^i S_T^\alpha(\theta)\r]^{p/2}&\leq& C\esp \l(\int_{[0,T] \times \mathbb{X}}{\partial_\theta\textnormal{log} q^\alpha(s,x,\theta)^2f^\alpha(s,x,\theta^*)\mathbb{1}_{\{\lambda^\alpha(s,\theta^{*}) \neq 0\}}ds\rho(dx)}\r)^{\frac{p}{2}}\\
				&+& C\esp \int_{[0,T] \times \mathbb{X}}{\l|\partial_\theta\textnormal{log} q^\alpha(s,x,\theta)\r|^pf^\alpha(s,x,\theta^*)\mathbb{1}_{\{\lambda^\alpha(s,\theta^{*}) \neq 0\}}ds\rho(dx)}\\
				&\leq& C T^{\frac{p}{2}-1}\esp \int_{[0,T] \times \mathbb{X}}{\l|\partial_\theta\textnormal{log} q^\alpha(s,x,\theta)\r|^p \lambda^\alpha(s,\theta^*) ^{\frac{p}{2}}\mathbb{1}_{\{\lambda^\alpha(s,\theta^{*}) \neq 0\}}q^\alpha(s,x,\theta^*)ds\rho(dx)}\\
				&+& C\esp \int_{[0,T] \times \mathbb{X}}{\l|\partial_\theta\textnormal{log} q^\alpha(s,x,\theta)\r|^pf^\alpha(s,x,\theta^*)\mathbb{1}_{\{\lambda^\alpha(s,\theta^{*}) \neq 0\}}ds\rho(dx)}\\
				  &=& O\l(T^{\frac{p}{2}}\r),
				\eeas
so that by (\ref{ineqSobolevS}) we get $\esp \sup_{\theta \in \Theta}\l| \frac{S_T^\alpha(\theta)}{T} \r|^p = O(T^{-p/2}) = o(T^{-\gamma p})$ since $\gamma <1/2$. Combined with (\ref{defYtilde2}), this readily gives (\ref{convY2Ytilde2}). Finally, applying Fubini's theorem to $\tilde{\mathbb{Y}}^{(2)}_T(\theta)$ along with \textbf{[A3]} shows the existence of $\mathbb{Y}^{(2)}(\theta)$ such that $T^\gamma \sup_{\theta \in \Theta} \| \tilde{\mathbb{Y}}_T^{(2)}(\theta) - \mathbb{Y}^{(2)}(\theta) \|_p \to 0$, which, by Lemma \ref{lemAppliSobolev} can be strengthened to $T^\gamma\|  \sup_{\theta \in \Theta} | \tilde{\mathbb{Y}}_T^{(2)}(\theta) - \mathbb{Y}^{(2)}(\theta) | \|_p \to 0$ since we also have by \textbf{[A3]} that $T^\gamma \sup_{\theta \in \Theta} \| \partial_\theta \tilde{\mathbb{Y}}_T^{(2)}(\theta) - \mathbb{U}(\theta) \|_p \to 0$ for some field $\mathbb{U}(\theta)$. This yields $T^\gamma \|\sup_{\theta \in \Theta} |\mathbb{Y}_T^{(2)}(\theta) - \mathbb{Y}^{(2)}(\theta)| \|_p \to 0$ by (\ref{convY2Ytilde2}).

\end{proof}

\begin{proof}[Proof of Lemma \ref{lemGamma}]
We prove the existence for $i \in \{1,2\}$ of $\Gamma^{(i)}$ such that $\sup_{\theta \in B_T} |\Gamma_T^{(i)}(\theta) - \Gamma^{(i)}| \to^\proba 0$ and $T^\gamma \l\| \Gamma_T^{(i)}(\theta^*) - \Gamma^{(i)} \r\|_{p} \to 0$. The case $i=1$ is a consequence of Lemma 3.12 and Lemma 3.15 in \cite{clinet2017statistical} along with assumptions \textbf{[A1]}-\textbf{[A4]}. We now turn to the case $i=2$. It is convenient to rewrite $\Gamma_T^{(2)}(\theta) = \sum_{\alpha=1}^d \l[\tilde{\Gamma}_T^{(2),\alpha} + M_T^{(2),\alpha}(\theta) +   R_T^{(2),\alpha}(\theta)\r]$ with 
\beas
\tilde{\Gamma}_T^{(2),\alpha} &=&  T^{-1} \int_{[0,T] \times \mathbb{X}}  \partial_\theta^2 \textnormal{log}q^\alpha(t,x, \theta^*)  q^\alpha(t,x,\theta^*)\rho(dx) \lambda^\alpha(s,\theta^*) ds ,
\eeas
\beas  
M_T^{(2),\alpha}(\theta) &=&- T^{-1} \int_{[0,T] \times \mathbb{X}} {\partial_\theta^2 \textnormal{log} q^\alpha(s,\theta)1_{\{q^\alpha(s,\theta^{*}) \neq 0\}}\tilde{M}^\alpha(ds,dx)}
\eeas 
and
\beas 
R_T^{(2),\alpha}(\theta) = - T^{-1} \int_{[0,T] \times \mathbb{X}} \{\partial_\theta^2 \textnormal{log} q^\alpha(s,\theta)  - \partial_\theta^2 \textnormal{log} q^\alpha(s,\theta^*)\} \mathbb{1}_{\{q^\alpha(s,x,\theta^{*}) \neq 0\}} \rho(dx)\lambda^\alpha(s,\theta^*)ds.\\
\eeas
Next, remark that there exists $\Gamma^{(2)}$ such that $T^\gamma\| \sup_{\theta \in \Theta} |\tilde{\Gamma}_T^{(2),\alpha} - \Gamma^{(2)}| \|^p \to 0$ by $\textbf{[A3]}$ and Lemma \ref{lemAppliSobolev}. Now we have
$$ \esp \sup_{\theta \in B_T}|M_T^{(2),\alpha}(\theta)|^p  \leq \esp \sup_{\theta \in \Theta}|M_T^{(2),\alpha}(\theta)|^p = O(T^{-p/2}) = o(T^{-\gamma p})$$ 
where the last inequality is a consequence of Lemma \ref{lemBDGmartingale} and \textbf{[A2]}. Since $R_T^{(2),\alpha}(\theta^*) = 0$, we have proved the second claim of the lemma. Now, using again that $R_T^{(2),\alpha}(\theta^*) = 0$, we have for $\theta \in B_T$
\beas 
|R_T^{(2),\alpha}(\theta)| &\leq& |\theta - \theta^*| \sup_{\theta \in \Theta} |\partial_{\theta}R_T^{(2),\alpha}(\theta)| \\
&\leq& \textnormal{diam}(B_T)  \sup_{\theta \in \Theta} |\partial_{\theta}R_T^{(2),\alpha}(\theta)|
\eeas 
and noticing that the derivative of the integrand in $R_T^{(2),\alpha}(\theta)$ is proportional to $\partial_\theta^3 \textnormal{log} q^\alpha(s,\theta)$, application of \textbf{[A2]} readily gives the domination $\esp \sup_{\theta \in \Theta} |\partial_{\theta}R_T^{(2),\alpha}(\theta)|^p \leq K$ for some $K>0$, which, combined with the  convergence $\textnormal{diam}(V_T) \to^\proba 0$, yields that $\sup_{\theta \in B_T} |R_T^{(2),\alpha}(\theta)| \to^\proba 0$, hence the first claim of the lemma. \\
\end{proof}

\begin{proof}[Proof of Lemma \ref{lemDelta}]
 Recall that $M^{T}$ is (by \textbf{[A2]}) an $\mathbb{L}^p$ integrable martingale with the following representation: 
\bea
M_u^{T} = \frac{1}{\sqrt T} \sum_{\alpha = 1}^d\int_{[0,uT] \times \mathbb{X}} \frac{\partial_\theta f^\alpha(t,x, \theta^*)}{f^\alpha(t,x,\theta^*)} \mathbb{1}_{\{f^\alpha(t,x,\theta^*) \neq 0\}} \tilde{M}^\alpha(dt,dx).
\eea
Accordingly, by Corollary VIII.3.24 p.476 from \cite{JacodLimit2003}, the claimed limit theorem will hold if we show $\langle M^T, M^T\rangle_u \to^\proba u \Gamma$ for all $u \in [0,1]$, and for some $\epsilon >0$, the Lindeberg condition 
\bea 
L_T = \frac{1}{T^2}\sum_{\alpha = 1}^d \int_{[0,T] \times \mathbb{X}} \frac{\partial_\theta f^\alpha(t,x, \theta^*)^{\otimes 2}}{f^\alpha(t,x,\theta^*)^2} \mathbb{1}_{\l\{ \l|\frac{\partial_\theta f^\alpha(t,x, \theta^*)}{f^\alpha(t,x,\theta^*)}\r| \geq \epsilon\sqrt{T} \r\}}  \mathbb{1}_{\{f^\alpha(t,x,\theta^*) \neq 0\}}  \nu^\alpha(dt,dx) \to^\proba 0. 
\eea 
Remark that by the ergodicity assumption \textbf{[A3]} (decomposing $f^\alpha = \lambda^\alpha \cdot q^\alpha$) and the fact that $\Delta N^\alpha \Delta N^\beta = 0$ for $\alpha \neq \beta$, 
\beas 
\langle M^{T}, M^T \rangle_u &=& \frac{1}{T} \sum_{\alpha=1}^d \int_{[0,uT] \times \mathbb{X}} \frac{\partial_\theta f^\alpha(t,x, \theta^*)^{\otimes 2}}{f^\alpha(t,x,\theta^*)} \mathbb{1}_{\{f^\alpha(t,x,\theta^*) \neq 0\}}  dt \rho(dx)  \to^\proba u \Gamma, 
\eeas
where the expression of the limit comes from (\ref{fisherRepresentation}). Moreover,
\beas 
\esp |L_T| \leq \frac{1}{T^{3/2} \epsilon} \esp \sum_{\alpha = 1}^d \int_{[0,T] \times \mathbb{X}} \frac{\partial_\theta f^\alpha(t,x, \theta^*)^{\otimes 2}}{f^\alpha(t,x,\theta^*)} \l|\frac{\partial_\theta f^\alpha(t,x, \theta^*)}{f^\alpha(t,x,\theta^*)}\r|  \mathbb{1}_{\{f^\alpha(t,x,\theta^*) \neq 0\}} dt \rho(dx) = O(T^{-1/2})
\eeas 
by \textbf{[A2]} and H\"{o}lder's inequality. Finally the second claim is an application of Burkholder-Davis-Gundy inequality along with H\"{o}lder's inequality and \textbf{[A2]}.
\end{proof}

\begin{proof}[Proof of Theorem \ref{thmLAN}]
We first prove \textbf{(i)}. The large deviation inequality is a consequence of Theorem 3 (c) in \cite{YoshidaPolynomial2011}: Setting $\beta_1 = \gamma$, $\beta_2 = 1/2- \gamma$, $\rho =2$, $\rho_2 \in (0,2\gamma)$, $\alpha\in (0,\rho_2/2)$ and $\rho_1 \in (0, \textnormal{min}(1,\alpha(1-\alpha)^{-1},2\gamma(1-\alpha)^{-1}))$, we immediately have condition (A4$'$). Moreover, (A6) and (A4$''$) are satisfied by Lemmas \ref{lemY}, \ref{lemGamma}, and \ref{lemDelta} along with the domination for all $p>1$
\bea \label{ineql3} 
\sup_{T \in \reels_+}\l\| T^{-1}\sup_{\theta \in \Theta} \l|\partial_{\theta}^3 l_T(\theta)\r|\r\|_p < +\infty
\eea 
which is yet another application of \textbf{[A2]}, Sobolev's inequality (involving the derivative of fourth order of $l_T$) and H\"{o}lder's inequality. Finally, conditions (B1) and (B2) are consequences of \textbf{[A3]} along with Remark \ref{rmkGamma}. Next we show \textbf{(ii)}. Rewrite 
$$ \textnormal{log}\mathbb{Z}_T(u) = u^T \Delta_T + u^T \Gamma_T(\theta^*) u + r_T(u) $$
where, by the mean value theorem, there exists $ \theta_u = w_u \theta^* + (1-w_u) (\theta^* + u/\sqrt{T})$ for some $w_u \in [0,1]$ such that 
$$r_T(u) = T^{-3/2}\partial_{\theta}^3 l_T(\theta_u)[u^{\otimes 3}]:= T^{-3/2} \sum_
{i,j,k=1}^{n} \partial_{\theta}^3 l_T(\theta_u)_{ijk} u_iu_ju_k.$$
For a fixed $u \in \reels^n$, an application of (\ref{ineql3}) immediately yields for any $p >1$ that $\esp |r_T(u)|^p \to 0$, which, combined with Lemmas \ref{lemGamma} and \ref{lemDelta}, the continuous mapping theorem, and Slutsky's lemma readily yields the finite dimensional convergence of $\mathbb{Z}_T$ towards $\mathbb{Z}$. Now we prove the tightness of $\textnormal{log }\mathbb{Z}_T$ in $T$. introducing for $\delta >0$, $r >0$, 
$$ w_T(\delta,r) = \sup_{u_1, u_2 \in B_{r,T}, |u_2-u_1|\leq \delta } |\textnormal{log} \mathbb{Z}_T(u_2) - \textnormal{log} \mathbb{Z}_T(u_1)|,$$
where $B_{r,T} = \{ u \in U_T| |u| \leq r\}$, it is sufficient to prove for some $p >n$, any $r>0$ the convergence $\lim_{\delta \to0}\sup_{T \in \reels_+} \esp[w_{T}(\delta,r)^p ] \to 0$ by \textbf{(i)}. We have 
\bea 
\nonumber \esp[w_{T}(\delta,r)^p ] &\leq& \esp \sup_{u_1, u_2 \in B_{r,T}, |u_2-u_1|\leq \delta } |l_T(\theta_{u_2}) - l_T(\theta_{u_1})|^p \\
&\leq& K \delta^p T^{-p/2} \esp \sup_{\theta \in \Theta} |\partial_\theta l_T(\theta)|^p \to 0 \label{eqmodulus}
\eea 
where the last convergence is again an application of Lemma \ref{lemBDG}, \textbf{[A2]}, and H\"{o}lder's inequality. Finally, \textbf{(iii)} is a consequence of Lemma 2 from \cite{YoshidaPolynomial2011} (taking $K(u) = \textnormal{log} \mathbb{Z}_T(u)$)   
along with (\ref{eqmodulus}).
\end{proof}

\begin{proof}[Proof of Theorem \ref{thmQLA2}]
For the QMLE, the convergence is a consequence of the LAN property and the large deviation inequality along with Theorem 4 (b) from \cite{YoshidaPolynomial2011}. Moreover, applying Theorem 8 in \cite{YoshidaPolynomial2011} along with the large deviation inequality and the inverse moment condition yields the convergence for any QBE.  
\end{proof}

\subsection{Proofs of Section \ref{sectionHawkes}}

Before we prove Proposition \ref{propKernel}, let us construct the matrices $A_u$ and $B_u$. Define first 
\beas  
\Pi = \l(\begin{matrix} r & 0 & \cdots & \cdots &  0 \\
                        1 & \ddots & \ddots& 0 & \vdots \\
                        0 & 2 & \ddots & \ddots & \vdots \\
                        \vdots & \ddots &\ddots &\ddots&0 \\
                        0& \cdots & 0 & p & r\end{matrix}\r) \in \reels^{(p+1) \times (p+1)}.
\eeas 
Then $B_u$ is defined as
\bea \label{defBu} 
B_u = \l(\begin{matrix} \Pi & 0 & 0 \\
                        0 & \Pi & \xi \cali_{p+1}\\
                        0&-\xi \cali_{p+1} & \Pi
                        \end{matrix}\r) \in \reels^{(3p+3)\times (3p+3)}.
\eea 
Now, let $\mathbf{a} = \l(a_0,...,a_p, ca_0,...,ca_p, da_0,...,da_p \r)^T,$ and $\mathbf{b}$ the column vector of $\mathbb{R}^{3p+3}$ such that $\mathbf{b}_1 = \mathbf{b}_{p+2} = 1$, and all its other components are null. Then we define 
\bea \label{defAu}  
A_u = \mathbf{a} \mathbf{b^T}. 
\eea 
\begin{proof}[Proof of Proposition \ref{propKernel}]
Assume that $f$ admits the claimed representation. By the Jordan normal form decomposition (see Theorem 3.4.1.5 in \cite{horn2012matrix}) of $B$, $B$ is similar to a block-diagonal matrix $\tilde{B} = \textnormal{diag}(J_1,...,J_q)$ where the so-called Jordan blocks $J_is$ are either of the form 
\beas 
J_i=\l(\begin{matrix}
\lambda_i & 1 & 0 &0  \\
0 & \ddots & \ddots & 0  \\
\vdots &\ddots&\ddots & 1 \\
0 & \cdots&0 &\lambda_i 
\end{matrix}\r)
\eeas 
with $\lambda_i$ being a positive real eigenvalue for $B$, or 
\beas 
J_i=\l(\begin{matrix}
C_i & \cali_2 & 0 &0  \\
0 & \ddots & \ddots & 0  \\
\vdots &\ddots&\ddots & \cali_2 \\
0 & \cdots&0 &C_i 
\end{matrix}\r) \textnormal{ where } C_i = \l(\begin{matrix}
\alpha_i & -\beta_i\\ \beta_i & \alpha_i 
\end{matrix}\r)
\eeas 
for some $\alpha_i, \beta_i \in \reels$, where in that case $\lambda_i = \alpha_i + \mathbf{i} \beta_i$ (with $\mathbf{i}^2=-1$) is a complex eigenvalue for $B$. Since by assumption all the real parts of these eignevalues are positive, then a direct calculation on the Jordan blocks $J_1,...,J_q$ readily shows using both examples (\ref{polynomExample}) and (\ref{cosExample}) that the coefficients of $e^{-tB}$ are of the form (\ref{canonicalFormKernel}), and we are done. Conversely, let us now prove that $u$ admits the exponential representation
\beas 
u = \langle A_u | e^{-\cdot B_u} \rangle
\eeas 
where $A_u$ and $B_u$ were respectively defined in (\ref{defAu}) and (\ref{defBu}). For $t \geq 0$, $k \in \{0,...,p\}$ define, $v_{k}(t) = t^k e^{-rt}$, $v_{k,c}(t) = t^k \textnormal{cos}(\xi t) e^{-rt}$, and $v_{k,s}(t) = t^k \textnormal{sin}(\xi t) e^{-rt}$, and $V = \l(v_0,...,v_p, v_{0,c},...,v_{p,c}, v_{0,s},...,v_{p,s}\r)^T$. We immediately check that $V$ satisfies the ordinary differential equation $V' = -B_{u} V$, and $V(0) = \textbf{b}$, so that $V(t) = e^{-tB_u}\textbf{b}$. Since $u = \textbf{a}^T V$, we get that $u(t) = \sum_{i,j=1}^{3p+3}\textbf{a}_i \textbf{b}_j [e^{-tB_u}]_{i,j} = \langle A_u |e^{-tB_u} \rangle$. Finally, one easily shows that the eigenvalues of $B_u$ are all in the set $\{r, r+\textbf{i}\xi, r -\textbf{i}\xi\}$ and thus have a negative real part.
\end{proof}
We now prove that $Z$ is a Markov process.

\begin{proof}[Proof of Proposition \ref{propGen}]
\textbf{Step 1:} We prove \textit{1-3}. By construction and (\ref{defHgenExp}), the counting process $N$ admits a piecewise deterministic stochastic intensity vector between two successive jump times. In other words, conditionally to $\calf_{T_{i-1}}$, $\Delta T_i$ is the next jump of an inhomogenous Poisson process with intensity $\bar{\lambda}(t) = \sum_{\alpha=1}^d \lambda^\alpha(t)$. Moreover, note that on the event $\{T_{i-1} < t \leq T_i\}$, we have $\lambda^\alpha(t) = \mu^\alpha(t - T_{i-1},Z_{T_{i-1}})$, hence $\Delta T_i$ follows the density $f(\cdot, Z_{T_{i-1}})$. Next, it is well-known that for $d$ independent inhomogenous Poisson processes, given $\Delta T_i$, the probability for the label $\kappa_i$ of being equal to $\alpha$ is $\lambda^\alpha( T_i)/ \sum_{\beta=1}^d \lambda^\beta(T_i)$, which gives us the second point. Finally, the third point is a consequence of (\ref{defKernelY}) and of the fact that $\proba[X_i \in A | \calf_{T_{i-1}}, \Delta T_i] = \sum_{\alpha=1}^d \proba[X_i \in A | \calf_{T_{i-1}}, \Delta T_i, \kappa_i = \alpha] \proba[\kappa_i =\alpha | \calf_{T_{i-1}}, \Delta T_i ] $  along with the second point of this proposition.\\
\textbf{Step 2:} We show that $Z$ is Markovian and admits the claimed generator. Note that $Z$ is piecewise deterministic between jump times. Then $Z$ will be Markovian if given $\calf_{T_{i-1}}$ the distribution of $(\Delta T_i, \Delta Z_{T_i})$ depends on $Z_{T_{i-1}}$ only. By the first point of the proposition, the marginal distribution of $\Delta T_i$ given $\calf_{T_{i-1}}$ is a function of $Z_{T_{i-1}}$ indeed. Moreover, note that by (\ref{defExcitationElementaire}), $\Delta Z_{T_i}$ depends on $X_i$ and $Z_{T_{i-1}}$ only, and therefore by the third point we deduce that the law of $\Delta Z_{T_i}$ given $\calf_{T_{i-1}}$ and $\Delta T_i$ depends on $Z_{T_{i-1}}$ only, which proves that $Z$ is Markovian. Moreover, the Feller property of $Z$ easily comes from the fact that all quantities involved in the distribution of $Z$ are continuous in the initial condition $z$ and the kernels $\calq_\beta$ are assumed Feller too. Now we turn to the expression of the generator $\call$. Given an initial condition $Z_0 = z = (\epsilon,x) \in \E \times \mathbb{X}$ Note that by definition of $\cale$, we have for any $\alpha,\beta \in \{1,...,d\}$ 
\beas 
\cale_{\alpha\beta}(t) = \epsilon + \int_{[0,t] \times \mathbb{X}}  g_{\alpha\beta}(x) \overline{N}_\beta(ds,dx) - B_{\alpha\beta} \int_0^t \cale_{\alpha\beta}(s)ds.
\eeas 
Considering now $t \geq 0$ and using the above integral representation for $\cale$, we have on the event $A_{t,\beta} = \{ 0 < T_1 \leq t < T_2, \kappa_1 = \beta\} $ that, for a smooth and bounded function $f$ 
\beas 
f(Z_t) = f(\cale_t, X_1) &=& f\l(\epsilon + \l(g_{\alpha\gamma}(X_1)\mathbb{1}_{\{\gamma = \beta\}} - B_{\alpha\gamma} \int_0^t \cale_{\alpha\gamma}(s)ds\r)_{\alpha,\gamma \in \{1,...,d\}}, X_1\r).
\eeas 
Moreover, on the event $B_t = \{ 0 \leq t < T_1\}$, we have 
\bea  \label{eqIIGen}
f(Z_t)&=& f\l(z\r) - \sum_{\alpha,\beta=1}^d \int_0^t \frac{\partial f}{\partial \epsilon_{\alpha\gamma}}(\tilde{Z}_{\alpha\gamma}(s)). B_{\alpha\gamma} \cale_{\alpha\gamma}(s)ds
\eea 
with $\tilde{Z}_{\alpha\gamma}(s) = \l( \epsilon  - B_{\alpha\gamma} \int_0^s \cale_{\alpha\gamma}(u)du, X_1\r)$. We now write, with $C_t = \l[\bigcup_{\beta=1}^d A_{t,\beta} \cup B_t\r]^c$ (where for an event $A$, $A^c$ stands for its complementary event), 
\beas
t^{-1}\l(\esp[f(Z_t)] - f(z)\r) &=& t^{-1} \sum_{\beta = 1}^d \esp\l[(f(Z_t) - f(z))\mathbb{1}_{A_{t,\beta}}\r] + t^{-1}\esp\l[(f(Z_t) - f(z))\mathbb{1}_{B_{t}} \r] + t^{-1} \esp\l[(f(Z_t) - f(z)) \mathbb{1}_{C_t}\r]\\
&=& I + II + III.
\eeas 
Now, $I$ can be rewritten
\beas 
I = t^{-1} \sum_{\beta = 1}^d \esp\l[\l( f\l(\epsilon + \l(g_{\alpha\gamma}(X_1)\mathbb{1}_{\{\gamma = \beta\}} - B_{\alpha\gamma} \int_0^t \cale_{\alpha\gamma}(s)ds\r)_{\alpha,\gamma \in \{1,...,d\}}, X_1\r) - f(z)\r)\mathbb{1}_{A_{t,\beta}}\r],
\eeas 
and since $N$ admits stochastic intensities with respect to Lebesgue measure, standard arguments yield that $\proba[ T_2 \leq t] = O(t^2)$, so that $A_{t,\beta}$ can be replaced by $\tilde{A}_{t,\beta} = \{0 < T_1 \leq t, \kappa_1 =\beta\}$ without affecting the limit in time of $I$. By (\ref{defKernelY}), we know that the law of $X_1$ given $\tilde{A}_{t,\beta}$ is given by $\calq_{\beta}(z, \cdot)$, and also that $\proba[\tilde{A}_{t,\beta}] = \int_0^t \mu_\beta(u,z) e^{- \int_0^u \mu(s,z)ds}du$, so that by taking taking the limit $t \to 0$ in $I$ we readily get  by the dominated convergence theorem
\bea 
I \to \sum_{\beta=1}^d \int_\mathbb{X} \l(f\l(\epsilon + \l( g_{\alpha\gamma}(y)\mathbb{1}_{\{\gamma = \beta\}} \r)_{\alpha,\gamma \in \{1,...,d\}}, y\r) -  f(z) \r)\calq_\beta(x,dy) \underbrace{\mu_{\beta}(0,z)}_{\psi_{\beta}(z)}.
\eea
Similarly, from (\ref{eqIIGen}) and $\proba[B_t] \sim 1 - \mu(0,z)t$ as $t \to 0$, we easily deduce that
\beas  
II &=& t^{-1} \esp \sum_{\alpha,\beta=1}^d \int_0^t \frac{\partial f}{\partial \epsilon_{\alpha\gamma}}(\tilde{Z}_{\alpha\gamma}(s)). B_{\alpha\gamma} \cale_{\alpha\gamma}(s)ds + o(1) \\
&\to& \sum_{\alpha,\beta=1}^d \frac{\partial f}{\partial \epsilon_{\alpha\gamma}}(z). B_{\alpha\gamma} \epsilon_{\alpha\gamma}.
\eeas  
Finally $\proba[C_t] \sim O(t^2)$, so that $III \to 0$.
\end{proof}

\begin{proof}[Proof of Lemma \ref{lemDrift}]
Let us note that $A_{\alpha\beta}$ can be assumed having non-zero coefficients only, without loss of generality. Indeed, we can always remove from $\cale_{\alpha\beta}$ the corresponding components which never appear in the expression of $\lambda^\alpha$ and do not play any role in the dynamics of $\overline{N}$. For now, we assume that the coefficients appearing in the drift function $V$, $a_{\alpha\beta}$ and $\eta$, are of the form $\lambda \tilde{a}_{\alpha\beta}$ and $\lambda \tilde{\eta}$ respectively, where $\lambda \in [0,1]$ is a parameter that we will have to adjust, and the coefficients $\tilde{a}_{\alpha\beta}$ and $\eta$ are assumed fixed and will be specified later, and are small enough so that $c_{\alpha\beta} := \langle \tilde{a}_{\alpha\beta} | \cali_{p\times p}\rangle \leq \overline{c}$ for all $\alpha,\beta$, and $\tilde{\eta} <\overline{c}$ where $\overline{c}$ was introduced in \textbf{[L2]}. Recalling that $V(z)  = e^{\sum_{\alpha,\beta=1}^d \langle a_{\alpha\beta}| \epsilon_{\alpha\beta}\rangle +  \eta f_X(x)}$, Proposition \ref{propGen} immediately yields
\beas 
\call V(z) &=&V(z) \l(\sum_{\beta=1}^d \psi_\beta(z) \int_\mathbb{X} \{e^{  \lambda \sum_{\alpha=1}^d c_{\alpha\beta} g_{\alpha\beta}(y) + \lambda \tilde{\eta} [f_X(y) - f_X(x)] }  -1\}  \calq_\beta(x,dy)  -\sum_{\beta,\gamma=1}^d \langle a_{\beta\gamma} |B_{\beta\gamma} \epsilon_{\beta\gamma} \rangle \r).
\eeas
Our goal is to linearize the first term in $\lambda$ in order to get rid of the exponential function. Accordingly, we define $\xi(x) = e^x - 1 - x$, and noticing that for $\lambda \in [0,1]$, $\xi(\lambda x) \leq \lambda^2 \xi(x) \leq \lambda^2 e^x$, we rewrite 
\beas
\int_\mathbb{X} \{e^{  \lambda \sum_{\alpha=1}^d c_{\alpha\beta} g_{\alpha\beta}(y)+ \lambda \tilde{\eta} [f_X(y) - f_X(x)] }  -1\} \calq_\beta(x,dy) &\leq& \lambda \int_\mathbb{X}   \l\{ \sum_{\alpha=1}^d c_{\alpha\beta} g_{\alpha\beta}(y)   +   \tilde{\eta} [f_X(y) - f_X(x)]\r\}\calq_\beta(x,dy) \\ &+& \lambda^2 \int_\mathbb{X}    \xi\l(\sum_{\alpha=1}^d c_{\alpha\beta} g_{\alpha\beta}(y) +   \tilde{\eta} [f_X(y) - f_X(x)]\r)  \calq_\beta(x,dy)\\
&\leq& \lambda \sum_{\alpha=1}^d c_{\alpha\beta} G_{\alpha\beta}(x) + \lambda \tilde{\eta}  \int_\mathbb{X}[f_X(y) - f_X(x)]  \calq_\beta(x,dy) \\
&+& \lambda^2  \underbrace{\int_\mathbb{X} e^{\overline{c} \l[\sum_{\alpha=1}^d  g_{\alpha\beta}(y) + f_X(y) - f_X(x) \r]}  \calq_\beta(x,dy)}_{ \leq M < +\infty},
\eeas  
where the domination of the term of order $\lambda^2$ is ensured by \textbf{[L2]}. Injecting the above calculation in $\call V$, and bounding the term $\psi_\beta(z)$ by $\nu_\beta(x) +  \sum_{\gamma=1}^d \langle A_{\beta \gamma} | \epsilon_{\beta\gamma}\rangle $ by \textbf{[L1]}, we get
\beas 
\frac{\call V(z)}{V(z)} &\leq& \lambda \sum_{\beta=1}^d \nu_\beta(x) \l[\sum_{\alpha=1}^d c_{\alpha\beta} G_{\alpha\beta}(x)\r]  + \lambda \sum_{\beta,\gamma=1}^d \l\langle \l.  \l[\sum_{\alpha=1}^d c_{\alpha\beta} G_{\alpha\beta}(x) \r]  A_{\beta\gamma} - B_{\beta\gamma}^T\tilde{a}_{\beta\gamma}\r| \epsilon_{\beta\gamma} \r\rangle \\ &+&  \lambda \tilde{\eta}\sum_{\beta=1}^d \psi_\beta(z) \int_\mathbb{X}[f_X(y) - f_X(x)]  \calq_\beta(x,dy)  + \lambda^2M\sum_{\beta=1}^d \psi_\beta(z). 
\eeas 
Consider now $\tilde{a}_{\alpha\beta} = \l(B_{\alpha\beta}^{-1}\r)^T A_{\alpha\beta} \kappa_\alpha$, where $\kappa$ was defined in \textbf{[L3]}. Remark that for an admissible value of $\epsilon_{\alpha\beta}$ of the form $\cale_{\alpha\beta}(t)$, we have 
\beas 
\langle \tilde{a}_{\alpha\beta} | \cale_{\alpha\beta}(t) \rangle &=& \kappa_\alpha   \int_{[0,t] \times \mathbb{X}}  \langle A_{\alpha\beta} |   B_{\alpha\beta}^{-1} e^{-(t-s)B_{\alpha\beta}} \rangle g_{\alpha\beta}(x)\overline{N}_\beta(ds,dx) \\
&=& \kappa_\alpha \int_{[0,t] \times \mathbb{X}}  \l\{ \langle A_{\alpha\beta} |   B_{\alpha\beta}^{-1} \rangle  + \int_0^{t-s} h_{\alpha\beta}(u)du \r\} g_{\alpha\beta}(x)\overline{N}_\beta(ds,dx) \geq 0.
\eeas  
Note also that we can always scale $\kappa$ by an arbitrary small number so that $c_{\alpha\beta} \leq \overline{c}$.
The second term in the above equation can thus be rewritten, using (\ref{repPhi}) 
\beas 
\lambda \sum_{\beta,\gamma=1}^d \l\langle \l.  \l[\sum_{\alpha=1}^d c_{\alpha\beta} G_{\alpha\beta}(x) \r]  A_{\beta\gamma} - B_{\beta\gamma}^T\tilde{a}_{\beta\gamma}\r| \epsilon_{\beta\gamma} \r\rangle &=& \lambda \sum_{\beta,\gamma=1}^d \l\langle \l.   \l[\sum_{\alpha=1}^d \kappa_{\alpha} \Phi_{\alpha\beta}(x) - \kappa_\beta\r]  A_{\beta\gamma} \r| \epsilon_{\beta\gamma} \r\rangle\\
&\leq& \lambda (\rho-1) \sum_{\beta,\gamma=1}^d \l\langle \l.   \kappa_\beta  A_{\beta\gamma} \r| \epsilon_{\beta\gamma} \r\rangle.
\eeas 
As for the first and the third terms, another application of \textbf{[L2]} respectively gives
\beas 
\lambda \sum_{\beta=1}^d\nu_\beta(x) \l[\sum_{\alpha=1}^d c_{\alpha\beta} G_{\alpha\beta}(x)\r] \leq  \lambda \max_{\alpha,\beta = 1,...,d} c_{\alpha\beta} v(x) u_X(x) 
\eeas 
for some function $v$ tending to $0$ when $|x| \to +\infty$, and 
\beas 
\lambda \tilde{\eta}\sum_{\beta=1}^d \psi_\beta(z) \int_\mathbb{X}[f_X(y) - f_X(x)]  \calq_\beta(x,dy) \leq - \delta \lambda \tilde{\eta} u_X(x) + \lambda \tilde{\eta} K \sum_{\beta,\gamma=1}^d \l\langle \l.    A_{\beta\gamma} \r| \epsilon_{\beta\gamma} \r\rangle 
\eeas 
for some constants $\delta>0$ and $K \geq 0$, by (\ref{driftConditionMark}) and (\ref{momentDriftExpo}). Overall, defining $\underline{\kappa} = \min_{\beta=1,...,d} \kappa_\beta >0$ and using the fact that there exists $Q_1, Q_2 \geq 0$ such that $\sum_{\alpha = 1}^d \nu_\alpha \leq Q_1 + Q_2 u_X$, we get 
\beas 
\frac{\call V(z)}{V(z)} \leq  \lambda^2 Q_1 M + \l[\lambda^2MQ_2+\lambda v(x) - \delta \lambda \tilde{\eta} \r]u_X(x) + \l[ \lambda(\rho-1)\underline{\kappa} + \lambda \tilde{\eta}K  \r] \sum_{\beta,\gamma=1}^d \l\langle \l.    A_{\beta\gamma} \r| \epsilon_{\beta\gamma} \r\rangle.
\eeas  
Next, taking $\tilde{\eta} = \lambda$, and then taking $\lambda$ small enough yields 
\bea \label{dominationLvv}
\frac{\call V(z)}{V(z)} \leq K'+ \lambda \l[ v(x) - \frac{\delta}{2} \tilde{\eta} \r]u_X(x) + \lambda \frac{(\rho-1)\underline{\kappa}}{2}   \sum_{\beta,\gamma=1}^d \l\langle \l.    A_{\beta\gamma} \r| \epsilon_{\beta\gamma} \r\rangle,
\eea
for some $K' >0$, and now using that $v(x) \to 0$, $u_X(x) \to + \infty$ when $|x| \to + \infty$, and $\sum_{\beta,\gamma=1}^d \l\langle \l.    A_{\beta\gamma} \r| \epsilon_{\beta\gamma} \r\rangle \to + \infty$ when $|\epsilon| \to +\infty$, we deduce that there exists a compact set $\overline{K} \subset \E \times \mathbb{X}$ such that for $z \notin \overline{K}$, we have  $\frac{\call V(z)}{V(z)} \leq - \tilde{\delta} <0$. Since the right-hand side of (\ref{dominationLvv}) and $V$ are bounded on $\overline{K}$, this yields for some $\tilde{L}\geq 0$ large enough
\beas 
\call V \leq - \tilde{\delta} V + \tilde{L}, 
\eeas 
which is the claimed result.
\end{proof}

The next two lemmas are auxiliary results used in the proof of Lemma \ref{lemND}.
\begin{lemma*} \label{lemDeterminant}
Let $n \in \naturels - \{0\}$, $(\alpha_{i,j})_{1 \leq i,j \leq n}$ a family of positive numbers, and  $f_1,...,f_n : \reels_+ \to \reels$ a family of functions. Define the matrix $M[f_1,...,f_n,t_1,...,t_n] = [\alpha_{i,j} f_j(t_i)]_{1 \leq i,j \leq n}$. If $(f_i)_{1 \leq i \leq n}$ is linearly independent then there exist $0 \leq t_1 \leq t_2 \leq ... \leq t_n < +\infty$ such that 
\beas 
\textnormal{det} M[f_1,...,f_n,t_1,...,t_n] \neq 0.
\eeas 
\end{lemma*}

\begin{proof}
Assume that $(f_i)_{1 \leq i \leq n}$ is linearly independent. We prove our claim by induction on $n \geq 1$. When $n=1$,  $\textnormal{det} M[f_1,t_1] = \alpha_{11}f_1(t_1)$ which is non-null as soon as there exists $t_1 \geq 0$ such that $f_1(t_1) \neq 0$, which is obviously true. Let $n \geq 1$ be such that the result holds. Let us assume that for any $t_1,...,t_{n+1}$, $\textnormal{det} M[f_1,...,f_{n+1},t_1,...,t_{n+1}] = 0$. Then, application of Laplace's formula yields
\beas
0=\textnormal{det} M[f_1,...,f_{n+1},t_1,...,t_{n+1}] = \sum_{j=1}^n (-1)^{j+1} \alpha_{1,j} f_j(t_1) \textnormal{det} M[f_1,... f_{j-1},f_{j+1},...f_{n+1},t_2,...,t_{n+1}],
\eeas 
and since the equality holds for any $t_1 \geq 0$, this proves that a linear combination of the $f_js$ is null, which implies that each coefficient should be $0$. Since $\alpha_{1,j} \neq 0$, $\textnormal{det} M[f_1,... f_{j-1},f_{j+1},...f_{n+1},t_2,...,t_{n+1}] =0$ for any $j$ and any $t_2,...,t_n$. But this is in contradiction with the induction hypothesis, which in turn proves the existence of $t_1,...,t_{n+1}$ such that $\textnormal{det} M[f_1,...,f_{n+1},t_1,...,t_{n+1}] \neq 0$. 
\end{proof}

\begin{lemma*} \label{lemGradientInvertible}
Define $\x^*$ as in the proof of Lemma \ref{lemND}, and $z^* = (0, x_0)$. Then for $T>0$ large enough, there exists $\bt^*$ such that $\sum_{i=1}^{p^2d^2} t_i^* < T$ and 
\beas
\textnormal{det}\nabla_{\bt} C(T, \bt^*,\x^*,z^*) \neq 0.
\eeas 
\end{lemma*}

\begin{proof}
For $\alpha, \beta \in \{1,...,d\}$, we have that 
\bea \label{gradRep1}
\frac{\partial C_{\alpha\beta}}{\partial t_i}(T, \bt, \x^*, z^*) = B_{\alpha\beta}\sum_{j= i \vee ((\beta-1)p^2d+1) }^{\beta p^2d} g_{\alpha\beta}(x_j^*) e^{-\l[T - \sum_{k=1}^j t_k\r]B_{\alpha\beta}}
\eea
if $1\leq i \leq  \beta p^2d$, and 
\bea \label{gradRep2}
\frac{\partial C_{\alpha\beta}}{\partial t_i}(T, \bt, \x^*, z^*) = 0
\eea 
otherwise. We now represent $C(T, \bt, \x^*, z^*)$ as a row vector of $\mathbb{R}^{p^2d^2}$ of the form 
\bea \label{linearizationC} 
C = \l[L(C_{11}), ..., L(C_{d1}), L(C_{12}),...., L(C_{d2}),..., L(C_{1d}),..., L(C_{dd})\r],
\eea 
where for a matrix $M \in \reels^{p\times p}$, $L(M) \in \reels^{p^2}$ is the row vector corresponding to the concatenation of the rows of $M$, and where we have omitted the dependency in $(T, \bt, \x^*, z^*)$ in (\ref{linearizationC}) for the sake of clarity. Therefore, we have 
\beas
\nabla_{\bt} C = \l[\begin{matrix}\frac{\partial C}{\partial t_1} \\ \vdots \\  \frac{\partial C}{\partial t_{p^2d^2}} \end{matrix}\r],
\eeas 
which is a matrix of size $N \times N$ with $N = p^2d^2$. Using the representations (\ref{gradRep1})-(\ref{gradRep2}) and elementary operations on the rows of $\nabla_{\bt} C$ yields the triangular form 
\beas 
\textnormal{det} \nabla_{\bt} C = \l| \begin{matrix} M_1 & \ast  &\cdots & \ast \\ 
                                                    0 & M_2 & \ddots & \vdots\\
                                                    \vdots & \ddots & \ddots & \ast \\
                                                    0 & \cdots & 0 & M_{d}
                                    \end{matrix} \r| = \prod_{\beta=1}^d \textnormal{det} M_\beta
\eeas 

where for $\beta \in \{1,...,d\}$, the matrix $M_\beta$ has dimension $p^2d \times p^2d$, and moreover we have 
\beas 
|\textnormal{det} M_\beta| = \l| \begin{matrix} g_{1\beta}(x_{j_\beta}^*) L\l( B_{1\beta} e^{-\l[T-\sum_{k=1}^{j_\beta} t_k\r]B_{1\beta}}\r) &\cdots& g_{d\beta}(x_{j_\beta}^*) L\l( B_{d\beta} e^{-\l[T-\sum_{k=1}^{j_\beta} t_k\r]B_{d\beta}}\r) \\
\vdots & \cdots & \vdots \\
 g_{1\beta}(x_{j_\beta + p^2d - 1}^*) L\l( B_{1\beta} e^{-\l[T-\sum_{k=1}^{j_\beta + p^2d - 1} t_k\r]B_{1\beta}}\r) &\cdots& g_{d\beta}(x_{j_\beta + p^2d - 1}^*) L\l( B_{d\beta} e^{-\l[T-\sum_{k=1}^{j_\beta + p^2d - 1} t_k\r]B_{d\beta}}\r) 
                                    \end{matrix} \r|,
\eeas
with $j_\beta= (\beta-1)p^2d+1$. We now prove that, if $T$ is taken large enough, then there exists $t_{(\beta-1) p^2 d + 1}^* <... < t_{\beta p^2d}^*$ such that $|\textnormal{det} M_\beta| \neq 0$. First, by Proposition \ref{propKernel} we immediately deduce that each column in the above determinant is of the form 
\beas 
\l[ \begin{matrix} g_{\alpha\beta}(x_{j_\beta}^*) f_{q, \alpha\beta}\l(T-\sum_{k=1}^{j_\beta} t_k\r) \\ \vdots \\ g_{\alpha\beta}(x_{j_\beta + p^2d-1}^*) f_{q, \alpha\beta}\l(T-\sum_{k=1}^{j_\beta + p^2d - 1} t_k\r) \end{matrix}\r]
\eeas
where for each $q \in \{1,...,p^2d\}$, $f_{q, \alpha\beta}$ is a linear combination of functions of the form $P_{q,\alpha\beta}(t) (1+ c_{q,\alpha\beta} \textnormal{cos}(\xi_{q,\alpha\beta}t) +d_{q,\alpha\beta} \textnormal{sin}(\xi_{q,\alpha\beta}t) + )e^{-r_{\alpha\beta}t}$, with $P_{q,\alpha\beta}$ a polynomial function, $r_{\alpha\beta} >0$, and where the family $\{f_{q,\alpha\beta}\}_{1 \leq q \leq p^2d, 1 \leq \alpha \leq d }$ may be assumed linearly independent up to a dimension reduction of $\cale$, following Remark \ref{rmkLinearlyIndependent}. Therefore, since $g_{\alpha\beta}(x_{k}^*) \geq \underline{g} >0$ for any $\alpha,\beta \in \{1,...,d\}$ and any $k \in \{1,...,p^2d^2\}$ by \textbf{[ND1]}, we can apply inductively Lemma \ref{lemDeterminant} to $M_\beta$ for $\beta \in \{1,...,d\}$ and take $T$ arbitrary large to get the existence of $\bt^*$ with $t_{1}^* \leq... \leq t_{p^2d^2}^*$ and $\sum_{i=1}^{p^2d^2} t_i^* < T$, such that $\textnormal{det} M_\beta \neq 0$ for all $\beta \in \{1,...,d\}$. Remark that we can assume the $t_i^*$s distinct since, by continuity of $M_\beta$ in $(t_i)_{i \in \{1,...,\beta p^2d\}}$, $\textnormal{det} M_\beta$ remains non-zero if the $t_{i}^*$s are slightly shifted. 
\end{proof}
We are now ready to prove Lemma \ref{lemND}.

\begin{proof}[Proof of Lemma \ref{lemND}]
Let $T > 0$, $A \in \mathbf{B}(\E)$, $B \in  \mathcal{X}$ and some $z \in \E \times \mathbb{X}$. Let us introduce the functions $C_{\alpha\beta} : \reels_+^{p^2d^2+1} \times \mathbb{X}^{p^2d^2} \times (\E \times \mathbb{X}) \to \reels^{p \times p}  $ defined by 
\bea 
C_{\alpha\beta}(T,\bt, \x , z) = e^{-TB_{\alpha\beta} }\epsilon_{\alpha\beta} + \sum_{i= (\beta-1)p^2d + 1}^{\beta p^2d} g_{\alpha\beta}(x_i) e^{-\l[T - \sum_{j=1}^i t_j\r]B_{\alpha\beta}},
\eea 
where $\bt = (t_1,...,t_{p^2d^2})$ and $\x = (x_1,...,x_{p^2d^2})$, that we gather in the global function $C = (C_{\alpha\beta})_{\alpha,\beta = 1,...,d}$. Introducing now the events 
\beas 
E_{T, \Delta \mathbf{T}, \Delta T_{p^2d^2+1}} = \{T_{p^2d^2} < T <  T_{p^2d^2+1}  \}, 
\eeas 
\beas 
F = \cap_{\beta=1}^d \{ \kappa_{(\beta-1)p^2d +1} = \kappa_{(\beta-1)p^2d +2} = ...=\kappa_{\beta p^2d} = \beta \},
\eeas 
where $\Delta \mathbf{T} = (\Delta T_1,...,\Delta T_{p^2d^2})$ corresponds to the random vector of the first $p^2d^2$ inter-arrival times, and noting that on $E_{T, \Delta \mathbf{T}} \cap F \cap \{Z_0 = z\}$ we have $\cale(T) = C(T, \Delta \mathbf{T}, \mathbf{X}, z)$ with $\mathbf{X} = (X_1,...,X_{p^2d^2})$ the vector of the first $p^2d^2$ marks, we have by definition of $P^T$ 
\bea \label{dominationKernel1}
\nonumber P^T(z,A \times B) &=& \proba\l[ Z_T \in A \times B | Z_0 = z\r] \\
&\geq& \proba\l[ \{  C(T, \Delta \mathbf{T}, \mathbf{X}, z) \in A \} \cap \{X_{p^2d^2} \in B\} \cap E_{T, \Delta \mathbf{T}, \Delta T_{p^2d^2+1}} \cap F |Z_0=z \r].
\eea 
Note now that the three events in the above conditional probability are completely determined by the random variable $(\Delta \mathbf{T}, \mathbf{X}, \Delta T_{p^2d^2+1})$, which admits a conditional distribution given $\{Z_0 = z\}$ which, jointly with the event $\cap_{\beta=1}^d \{ \kappa_{(\beta-1)p^2d +1} = \kappa_{(\beta-1)p^2d +2} = ...=\kappa_{\beta p^2d} = \beta \}$, has the form
\beas 
\gamma( \bt, \x,t_{p^2d^2+1}, z) dt_1...dt_{p^2d^2+1} \prod_{i=1}^{p^2d^2} \calq_{\beta_i}(x_{i-1},dx_i) ,
\eeas 
where $x_0 = x$ and $\beta_i = \beta$ if and only if $i \in \{(\beta-1)p^2d + 1,..., \beta p^2d\} $, and where $\gamma$ is defined by 
\beas 
\gamma( \bt, \x,t_{p^2d^2+1}, z) = f(t_{p^2d^2+1}, z_{p^2d^2})\prod_{i=1}^{p^2d^2} \frac{\mu^{\beta_i}(t_i, z_{i-1})}{\mu(t_i, z_{i-1})}\cdot f(t_i,z_{i-1}),
\eeas 
with $z_0 = z$, and $z_{i} = (\epsilon_i, x_i)$, $[\epsilon_i]_{\alpha\beta} = e^{-t_i B_{\alpha\beta}} [\epsilon_{i-1}]_{\alpha\beta} + g_{\alpha\beta}(x_i)\mathbb{1}_{\{\beta=\beta_i\}} \cali_{p\times p} $ for $\alpha,\beta \in \{1,...,d\}$, by a repeated application of Proposition \ref{propGen}. The lower bound (\ref{dominationKernel1}) can therefore be rewritten, applying \textbf{[ND2]} 
\beas 
P^T(z,A \times B) \geq \int_{\U} \mathbb{1}_{\{ C(T, \bt, \x, z) \in A\}} \mathbb{1}_{\{ x_{p^2d^2} \in B\}} \mathbb{1}_{E_{T, \bt, t_{p^2d^2+1}}}\gamma( \bt, \x,t_{p^2d^2+1}, z) dt_1...dt_{p^2d^2+1} \prod_{i=1}^{p^2d^2} \calt_{\beta_i}(x_{i-1},dx_i),  
\eeas 
where $\U = \reels_+^{p^2d^2+1} \times \mathbb{X}^{p^2d^2}$. Now, by \textbf{[ND2]}, and since for any $x \in \mathbb{X}$, we have $\int_{\mathbb{X}}r_\beta(x,y)\sigma_\beta(dy) > 0$, we readily construct a sequence $\x^*$ such that $r_{\beta_1}(x_0,x_1^*) >0$, and then $r_{\beta_{i+1}}(x_{i-1}^*, x_i^*) > 0$ for $2 \leq i \leq p^2d^2$. By Lemma \ref{lemGradientInvertible}, define now $(\bt^*, \x^*, z^*)$ where $z^* = (x_0,0)$, and $\x^*$ is as above, and such that $\nabla_{\bt} C(T, \bt^*,\x^*,z^*)$ is invertible, and $\sum_{i=1}^{p^2d^2} t_i^* < T$. By Lemma 6.2 from \cite{benaim2015qualitative}, we may assume that there exists a bounded neighborhood $J \subset \mathbb{X}^{p^2d^2} \times (\E \times \mathbb{X})$ of $(\x^*, z^*)$  and of the form $J = J_1 \times ...\times J_{p^2d^2} \times (J_\epsilon \times J_0)$ where each component in the product is a neighborhood of the related component of $(\x^*, z^*)$, and such that for all $(\x,z) \in J$, there exists a neighborhood $W_{(\x,z)}$ of $\bt^*$ and an open set $I \subset \mathbb{R}^{p^2d^2}$ such that the restriction of $\bt \to C(T, \bt,\x,z)$ on $W_{(\x,z)}$ onto $I$, denoted by $\tilde{C}(T, \bt,\x,z)$, is a diffeomorphism. Moreover, there exists a neighborhood $W$ of $\bt^*$ such that $W_{(\x,z)} \subset W$ for any $(\x,z) \in J$. Without loss of generality, and since for any $T >0,t_{p^2d^2+1} >0$ the set of $\bt$ satisfying $E_{T, \bt, t_{p^2d^2+1}}$ is open, we may further assume that $W \subset E_{T, \bt, t_{p^2d^2+1}}$. Now, using the fact that $\gamma $ is positive and continuous and $\nabla_{\bt} C(T,\bt,\x,z)$ is non-singular on $W$ we have 
$$ \inf_{(\bt,\x,z) \in W \times J} \gamma(\bt,\x,t_{p^2d^2+1},z) \textnormal{det}\l(\nabla_{\bt} C(T,\bt,\x,z)\r) > 0$$
which yields for $z \in J_{\epsilon} \times J_0$, by the change of variable $\boldv = \tilde{C}(T, \bt, \x, z)$
\beas 
P^T(&z,&A \times B)\\ &\geq&  c\int_{\U} \mathbb{1}_{\{ \tilde{C}(T, \bt, \x, z) \in A\}} \mathbb{1}_{\{ (\bt,\x,z) \in W_{\x} \times J\}} \mathbb{1}_{\{ x_{p^2d^2} \in B\}} \textnormal{det}\l(\nabla_{\bt} C(T,\bt,\x,z)\r)^{-1} dt_1...dt_{p^2d^2+1} \prod_{i=1}^{p^2d^2} \calt_{\beta_i}(x_{i-1},dx_i)\\
&\geq& c\int_{\U} \mathbb{1}_{\{ \boldv \in A\}}\mathbb{1}_{\{(\boldv,\x,z) \in I \times J\}} d\boldv dt_{p^2d^2+1} \mathbb{1}_{\{ x_{p^2d^2} \in B\}} \prod_{i=1}^{p^2d^2+1} \calt_{\beta_i}(x_{i-1},dx_i)\\
&=& c \textnormal{Leb}(I \cap A) \int_{J_1 \times...\times J_{p^2d^2} \times (J_{p^2d^2} \cap B)} r_{\beta_1}(x_0,x_1)...r_{\beta_{p^2d^2}}(x_{p^2d^2-1},x_{p^2d^2}) \sigma_{\beta_1}(dx_1)...\sigma_{\beta_{p^2d^2}}(dx_{p^2d^2} )
\eeas 
for some $c>0$. Up to a further reduction of the size of the neighborhoods $J_0,J_1,...,J_{p^2d^2}$ we may assume that $r_{\beta_i}(x_{i-1},x_i) \geq \underline{r} >0$ for some $\underline{r}$ independent of $i$, on $J_{i-1} \times J_i$, by lower semi-continuity of $r_{\beta_i}$ and the fact that $r_{\beta_i}(x_{i-1}^*,x_i^*) > 0$ with the convention $x_0^* = x_0$. This yields 
\beas 
\inf_{z \in J_\epsilon \times J_0} P^T(z,A \times B) &\geq& c \textnormal{Leb}(I \cap A) \underline{r}^{p^2d^2} \prod_{i=1}^{p^2d^2}\sigma_{\beta_i}(J_i) \sigma_{\beta_{p^2d^2+1}}(J_{p^2d^2+1} \cap B),
\eeas 
where $\textnormal{Leb}$ stands for the Lebesgue measure and $\sigma_{\beta_{p^2d^2+1}}(J_{p^2d^2+1} \cap \cdot)$ is non-trivial by \textbf{[ND2]}. This proves that $U =  J_{\epsilon} \times J_0$ is small for $P^T$. Finally, $U$ is accessible. Indeed, $J_0$ is accessible for $\calq$ by \textbf{[ND2]} (because $x_0\in J_0$ and is reachable for $\calq$), and then $\epsilon^* = 0$ is clearly a reachable point for $\cale$, since $\cale_t \to 0$ for $t \to+\infty$ on the event where there are no jumps after a given time $\overline{t}$ and up to $t$, so that $J_\epsilon$ can be visited by $\cale$ once $X_t$ has reached (and stays in) $J_0$. Finally, the fact that compact sets are petite is a consequence of the Feller property of $Z$ along with Theorem 12.1.10 from \cite{douc2018markov} taking $U$ as the open petite set (note that $P^T$ admits an accessible small set, $U$, and therefore is irreducible as required). 
\end{proof}

\begin{proof}[Proof of Theorem \ref{thmVgeom}]
This is an immediate consequence of the drift condition obtained in Lemma \ref{lemDrift}, the fact that compacts sets are petite by Lemma \ref{lemND}, and Theorem 6.1 from \cite{MeynTweedieprocessiii1993}. Finally the $V$-geometric mixing property is a consequence of
the proof of Theorem 16.1.5 p.398 in \cite{MeynTweedieMarkovChain2009}.
\end{proof}

\begin{proof}[Proof of Corollary \ref{corolVgeom}]

First remark that, by \textbf{[ND1]}, $g_{\alpha\beta}>0$ and therefore the counting measure associated to the jumps of $\cale$ is exactly $N = \overline{N}(\cdot \times \mathbb{X})$. By Theorem \ref{thmVgeom}, $Z$ admits an invariant probability $\pi$. By Theorem 3.1.7 from \cite{douc2018markov} adapted to the case where the time index is $\reels_+$, we readily obtain the existence of a stationary two-sided process $Z'$ on $\reels$ with marginal distribution $\pi$ and generator $\call$, possibly on a larger probability space. This yields existence of a stationary counting process $N'$, and defining for any $\alpha \in \{1,...,d\}$ $\overline{N}^\alpha(ds,dx) = \sum_{i \in \mathbb{Z}} \delta_{( T_i^\alpha, X'_{T_i^\alpha})}(ds,dx)$, where $\{T_i\}_{i \in \mathbb{Z}}$ are the jump times of $N'^\alpha$. By construction, the stochastic intensity $\lambda'$ of $N'$ admits the representation for any $u \leq t$
\beas 
 \lambda'^\alpha(t) = \phi_\alpha \l( \l(\langle A_{\alpha\beta} |e^{-(t-u)B_{\alpha\beta}}\cale_{\alpha\beta}'(u) \rangle + \int_{[u,t) \times \mathbb{X}} \langle A_{\alpha\beta} | e^{-(t-s) B_{\alpha\beta}} \rangle g_{\alpha\beta}(x)\overline{N}_\beta'(ds,dx)\r)_{\beta =1,...,d},X_{t-}'\r)
\eeas 
and taking the limit $u \to -\infty$, and using the stationarity of $\cale'$ yields $\langle A_{\alpha\beta} |e^{-(t-u)B_{\alpha\beta}}\cale_{\alpha\beta}'(u) \rangle \to^\proba 0$, and by continuity of $\phi_\alpha$ in its first argument we get the claimed representation (\ref{eqInfinity}).
\end{proof}

\subsection{Proofs of Section \ref{sectionQLAHawkes}}

\begin{lemma*} \label{lemMomentHawkes}
Assume the conditions of Corollary \ref{thmFinal}. Then \textnormal{\textbf{[A2]}} holds.
\end{lemma*}

\begin{proof}
Given \textbf{[AH2]} and the $V$-geometric ergodicity of $Z$, note that we directly have 
$$\sup_{t \in \reels_+} \sum_{i=0}^3  \esp \l[\sup_{\theta \in \Theta} |\partial_\theta^i g_{\alpha\beta}(X_{t-}, \theta)|^p  \r]< +\infty,$$
$$\sup_{t \in \reels_+}  \sum_{i=0}^3 \int_\mathbb{X} \esp \l[\sup_{\theta \in \Theta} \l|\partial_\theta^i \textnormal{log}p_\alpha(X_{t-},x,\theta)\r|^p  p_\alpha(X_{t-},x,\theta^*)\r] \rho(dx)  <+\infty,$$
$$\sup_{t \in \reels_+}  \sum_{i=0}^3 \int_\mathbb{X} \esp \l[\sup_{\theta \in \Theta} \l|\partial_\theta^i \textnormal{log}p_\alpha(X_{t-},x,\theta)\r|^{-p} p_\alpha(X_{t-},x,\theta^*) \r] \rho(dx)  <+\infty,$$
and 
$$\sup_{t \in \reels_+} \esp\l[ \sup_{\theta \in \Theta} \nu_\alpha(X_{t-},\theta)^{-p} \mathbb{1}_{\{\nu_\alpha(X_{t-},\theta) \neq 0\}}\r] < +\infty,$$
which readily yields \textbf{[A2](ii)-(iv)}, and so only \textbf{(i)} remains to be proved. For $i \in \{0,...,4\}$, recall that we have 
\beas 
|\partial_\theta^i \lambda^\alpha(t,\theta)| \leq \nu_\alpha(X_{t-},\theta) + \sum_{\beta=1}^d \int_{[0,t) \times \mathbb{X}} \l|\partial_\theta^{i} h_{\alpha\beta}(t-s, x, \theta)\r| \overline{N}_\beta(ds, dx)
\eeas 
and since $\sup_{t \in \reels_+} \esp \sup_{\theta \in \Theta}|\nu(X_{t-},\theta) |^p  <+\infty$  by \textbf{[AH2]} and the $V$-geometric ergodicity of $Z$, we have to prove is that for any $\alpha,\beta \in \{1,...,d\}$, any $p \geq 1$, any $t \geq 0$, 
\bea \label{eqSupA2}
\esp \l| \int_{[0,t) \times \mathbb{X}}\sup_{\theta \in \Theta} |\partial_\theta^{i}  h_{\alpha\beta}(t-s, x, \theta)| \overline{N}_\beta(ds, dx)\r|^p <+\infty.
\eea
Note that by Proposition \ref{propKernel}, and the fact that the real parts of all eigenvalues of $B_{\alpha\beta}$ are larger than $r>0$ uniformly in $\theta$, we deduce that for any $\theta \in \Theta$
\bea \label{eqDevhg} 
\l|\partial_\theta^{i}  h_{\alpha\beta}(t, x, \theta)\r| \leq M e^{-\frac{r}{2}t} \sum_{j=0}^i \l|\partial_\theta^j g_{\alpha\beta}(x,\theta)\r|.
\eea 
Now, by Lemma \ref{lemBDG}, we have 
\beas 
\esp \l| \int_{[0,t) \times \mathbb{X}} e^{-\frac{r}{2}(t-s)}\sup_{\theta \in \Theta} \l|\partial_\theta^j g_{\alpha\beta}(x,\theta)\r| \overline{N}_\beta(ds, dx)\r|^p &\leq& C(I + II + III)
\eeas 
for some $C>0$ with 
\beas 
I &=&  \esp \l| \int_{[0,t) \times \mathbb{X}}e^{-\frac{r}{2}(t-s)}\sup_{\theta \in \Theta} \l|\partial_\theta^j g_{\alpha\beta}(x,\theta)\r| \lambda_\beta(s,\theta^*) q^\beta(s,x,\theta^*) ds \rho(dx)\r|^p \\
II &=&  \esp  \int_{[0,t) \times \mathbb{X}}e^{-\frac{rp}{2}(t-s)}\sup_{\theta \in \Theta} \l|\partial_\theta^j g_{\alpha\beta}(x,\theta)\r|^p \lambda_\beta(s,\theta^*) q^\beta(s,x,\theta^*) ds \rho(dx) \\
III &=&  \esp \l| \int_{[0,t) \times \mathbb{X}}e^{-r(t-s)}\sup_{\theta \in \Theta} \l|\partial_\theta^j g_{\alpha\beta}(x,\theta)\r|^2 \lambda_\beta(s,\theta^*) q^\beta(s,x,\theta^*) ds \rho(dx)\r|^{p/2}.
\eeas 
We prove that $I$ is dominated by a constant uniformly in $t \in \reels_+$. The cases of $II$ and $III$ follow the same line of reasoning. When $t=0$, we directly have $II=0$. Assume therefore and without loss of generality that $t>0$. Define the probability measure $\eta(ds,dx) = \l(\int_0^t e^{\frac{r}{2}s}ds\r)^{-1}e^{\frac{r}{2}s}q^\beta(s,x,\theta^*)ds\rho(dx)$ on $[0,t] \times \mathbb{X}$. Applying Jensen's inequality to $\eta$, we get
\beas 
I &=&  \l(\int_0^t e^{\frac{r}{2}s}ds\r)^{p} \esp \l| \int_{[0,t) \times \mathbb{X}}e^{-\frac{rt}{2}}\sup_{\theta \in \Theta} \l|\partial_\theta^j g_{\alpha\beta}(x,\theta)\r| \lambda_\beta(s,\theta^*) \eta(ds,dx)\r|^p\\
&\leq&   \l(\int_0^t e^{\frac{r}{2}s}ds\r)^{p}  \esp\int_{[0,t) \times \mathbb{X}}e^{-\frac{rpt}{2}}\sup_{\theta \in \Theta} \l|\partial_\theta^j g_{\alpha\beta}(x,\theta)\r|^p \lambda_\beta(s,\theta^*)^p \eta(ds,dx)\\
&\leq&   \underbrace{\l(\int_0^t e^{\frac{r}{2}(s-t)}ds\r)^{p-1}}_{\leq M < +\infty}  \int_{[0,t) \times \mathbb{X}} e^{-\frac{r}{2}(t-s)}\esp \sup_{\theta \in \Theta} \l|\partial_\theta^j g_{\alpha\beta}(x,\theta)\r|^p \lambda_\beta(s,\theta^*)^p q^\beta(s,x,\theta^*) ds \rho(dx).\\
\eeas 
Now, recall that $\lambda_\beta(\cdot,\theta^*)$ is a sub-linear combination of the components of $\cale$, and so by the $V$-geometric ergodicity of Theorem \ref{thmVgeom} with $V$ being exponential in $\epsilon$, we readily get that $\sup_{t \in \reels_+} \esp[\lambda^\beta(t,\theta^*)]^q < +\infty$ for any $q \geq 1$. Combined with Cauchy-Schwarz inequality and \textbf{[AH2]-(iii)}, and then Jensen's inequality with respect to the probability measure $q^\beta(s,x,\theta^*)\rho(dx)$ we obtain for any $s \in [0,t]$
\beas
\esp  \int_{\mathbb{X}} \sup_{\theta \in \Theta} |\partial_\theta^j g_{\alpha\beta}&(x,\theta)&|^p \lambda_\beta(s,\theta^*)^p q^\beta(s,x,\theta^*)  \rho(dx) 
\\&\leq& \sqrt{\esp \lambda_\beta(s,\theta^*)^{2p} }\sqrt{ \esp \l(\int_{\mathbb{X}} \sup_{\theta \in \Theta} \l|\partial_\theta^j g_{\alpha\beta}(x,\theta)\r|^p q^\beta(s,x,\theta^*)  \rho(dx)\r)^{2} }\\
&\leq& C \sqrt{ \esp  \int_{\mathbb{X}} \sup_{\theta \in \Theta} \l|\partial_\theta^j g_{\alpha\beta}(x,\theta)\r|^{2p} q^\beta(s,x,\theta^*)  \rho(dx)  }\\
&\leq& C',
\eeas 
for two constants $C,C' >0$. Back to the expression of $I$, we readily get the uniform boundedness of $I$ with respect to $t \in \reels_+$.
\end{proof}

\begin{lemma*} \label{lemErgoHawkes}
Assume the conditions of Corollary \ref{thmFinal}. Then \textnormal{\textbf{[A3]}} holds for any $\gamma \in (0,1/2)$, with, for any $\alpha \in \{1,...,d\}$, $\theta \in \Theta$, and $\phi \in D_\uparrow(E,\reels)$  
$$ \pi_\alpha(\phi,\theta) = \esp\l[\phi(\lambda'^\alpha(0,\theta^*),\lambda'^\alpha(0,\theta), \partial_\theta \lambda'^\alpha(0,\theta) )\r]$$
and
$$\chi_\alpha(k,\theta) = \esp\l[ \int_{\mathbb{X}}  \partial_\theta^k  \textnormal{log}q'^\alpha(0,x,\theta)q'^\alpha(0,x,\theta^*)\rho(dx) \lambda'^\alpha(0,\theta^*)\r]$$
where $(\lambda',q')$ are the stochastic intensity and the mark density of the stationary version $\overline{N}'$ defined in Corollary \ref{corolVgeom}. 
\end{lemma*}

\begin{proof}
Fix some $\gamma \in (0,1/2)$. Introduce for $\alpha \in \{1,...,d\}$ the processes
\beas  
Y_1^\alpha(t,\theta) = \l(\lambda^\alpha(t,\theta^*),\lambda^\alpha(t,\theta),\partial_\theta \lambda^\alpha(t,\theta) \r),
\eeas  
\beas 
Y_2^\alpha(t,\theta) =\int_\mathbb{X} \partial_\theta^k  \textnormal{log}q^\alpha(t,x,\theta)q^\alpha(t,x,\theta^*)\rho(dx) \lambda^\alpha(t,\theta^*).
\eeas 
We are going to prove that there exists $\pi_\alpha$ as in \textbf{[A3]} such that for any $\phi \in D_\uparrow(E,\reels)$ , we have 
\bea \label{eqLLNproof1}
\sup_{\theta \in \Theta} T^{\gamma} \l\|T^{-1} \int_0^T \phi(Y_1^\alpha(s,\theta))ds - \pi_\alpha(\phi,\theta)\r\|_p \to 0.
\eea 
The case for $Y_2^{\alpha}$ follows a similar path, using that $q$ can be written as a function of $\lambda$ and $X$ by (\ref{repqparam}). By Lemma 3.16 and assumption \textbf{[M2]} from \cite{clinet2017statistical}, it is sufficient to prove that
for any $\phi,\psi \in  D_\uparrow(E,\reels)$ we have the mixing property 
\bea \label{eqLLNproof12}
\textnormal{Cov}\l[\phi(Y_1^\alpha(t,\theta)), \psi(Y_1^\alpha(t,\theta))\r] \leq C e^{-\tilde{r}u}
\eea 
for some $\tilde{r} > 2\gamma/(1-2\gamma)$, along with 
\bea \label{eqLLNproof13}
T^\gamma \l(\esp[\phi(Y_1^\alpha(t,\theta))] - \pi_\alpha(\phi,\theta)\r) \to 0
\eea 
in order to establish (\ref{eqLLNproof1}). Proving (\ref{eqLLNproof12})-(\ref{eqLLNproof13}) boils down to following the same proof as that of Lemma A.6 in \cite{clinet2017statistical}, replacing the truncated process $\tilde{X}$ in the original proof by the new expression
\beas 
\tilde{X}(s,t,\theta) = (\lambda^\alpha(t,\theta^*), \tilde{X}_1(s,t,\theta), \tilde{X}_2(s,t,\theta))
\eeas
with 
\beas 
\tilde{X}_1(s,t,\theta)  = \phi_\alpha\l( \l(\int_{[s,t) \times \mathbb{X}} h_{\alpha\beta}(t-s, x, \theta) \overline{N}_\beta(ds, dx)\r)_{\beta =1,...,d}, X_{t-}, \theta\r), 
\eeas
and
\beas 
\tilde{X}_2(s,t,\theta)  = \partial_\theta \tilde{X}_1(s,t,\theta). 
\eeas
Now, following closely the reasoning of the proof of Lemma A.6 in \cite{clinet2017statistical}, using that the kernels $h_{\alpha\beta}$ and $\partial_\theta h_{\alpha\beta}$ are exponentially decreasing in time, and replacing $\cale$ by $Z$ in all conditional expectations and applying Theorem \ref{thmVgeom} for the $V$-geometric ergodicity of $Z$ yields (\ref{eqLLNproof12}). Similar arguments yield (\ref{eqLLNproof13}) where 
$$\pi_\alpha(\phi,\theta) = \esp[\phi(\lambda'^\alpha(0,\theta^*),\lambda'^\alpha(0,\theta), \partial_\theta \lambda'^\alpha(0,\theta) )].$$
\end{proof}

We are now ready to prove Lemma \ref{lemLimitYHawkes} and Theorem \ref{thmFinal}.

\begin{proof}[Proof of Lemma \ref{lemLimitYHawkes}.]
By \textbf{[AH1]-[AH3]} and Lemmas \ref{lemErgoHawkes}-\ref{lemMomentHawkes}, \textbf{[A1]-[A3]} are satisfied so that Lemma \ref{lemY} holds. The shape of the limit field $\mathbb{Y}$ is then an immediate consequence of \textbf{[A3]} and the shape of $\pi_\alpha$ and $\chi_\alpha$ in Lemma \ref{lemErgoHawkes}. 
\end{proof}

\begin{proof}[Proof of Corollary \ref{thmFinal}]
By Lemma \ref{lemErgoHawkes}, Lemma \ref{lemMomentHawkes} and \textbf{[AH1]-[AH3]}, \textbf{[A1]-[A4]} hold and therefore we conclude by application of Theorem \ref{thmQLA2}. 
\end{proof}
\bibliography{biblio}
\bibliographystyle{apalike} 

\end{document}